\documentclass[twoside]{article}

\usepackage[accepted]{aistats2022}
\usepackage{amsmath, amssymb, amsthm}
\usepackage{float, graphicx, caption}
\usepackage{xcolor, subfigure, titlesec}
\usepackage{natbib}
\PassOptionsToPackage{hyphens}{url}\usepackage{hyperref}
\usepackage{xurl}

\hypersetup{
  breaklinks=true,
  colorlinks=true,
  linkcolor=black,
  urlcolor=blue,
  citecolor=black,
  anchorcolor=black,
  filecolor=black,
}

\DeclareMathOperator*{\minimize}{minimize}
\DeclareMathOperator*{\argmin}{argmin}

\DeclareMathOperator*{\st}{subject\;to}
\def\begm{\ensuremath\begin{bmatrix}}
\def\endm{\ensuremath\end{bmatrix}}

\newcommand\zs{\ensuremath {{z^\star}}}

\newcommand{\R}[1]{\ensuremath {{\mathbb{R}^{#1}}}}
\newcommand\norm[1]{\ensuremath {{\left|\left|#1\right|\right|}}}

\newcommand\Oh{{$\mathcal{O}$}}

\newtheorem{thm}{Theorem}

\begin{document}

\runningtitle{Second-Order Sensitivity Analysis for Bilevel Optimization}

\runningauthor{Robert Dyro, Edward Schmerling, Nikos Ar\'echiga, Marco Pavone}

\twocolumn[
\aistatstitle{Second-Order Sensitivity Analysis for Bilevel Optimization}
\vspace*{-0.7cm}
\aistatsauthor{Robert Dyro \And Edward Schmerling \And Nikos Ar\'echiga \And Marco Pavone}
\aistatsaddress{Stanford University \\ \texttt{rdyro@stanford.edu}
  \And  Stanford University \\ \texttt{schmrlng@stanford.edu}
  \And Toyota Research Institute \\ \texttt{nikos.arechiga@tri.global}
\And Stanford University \\ \texttt{pavone@stanford.edu} }
]
\vspace*{-1.0cm}

\begin{abstract}
  \vspace*{-0.1cm}
In this work we derive a second-order approach to bilevel optimization, a type of mathematical programming in which the solution to a parameterized optimization problem (the ``lower'' problem) is itself to be optimized (in the ``upper'' problem) as a function of the parameters. Many existing approaches to bilevel optimization employ first-order sensitivity analysis, based on the implicit function theorem (IFT), for the lower problem to derive a gradient of the lower problem solution with respect to its parameters; this IFT gradient is then used in a first-order optimization method for the upper problem. This paper extends this sensitivity analysis to provide second-order derivative information of the lower problem (which we call the IFT Hessian), enabling the usage of faster-converging second-order optimization methods at the upper level. Our analysis shows that (i) much of the computation already used to produce the IFT gradient can be reused for the IFT Hessian, (ii) errors bounds derived for the IFT gradient readily apply to the IFT Hessian, (iii) computing IFT Hessians can significantly reduce overall computation by extracting more information from each lower level solve. We corroborate our findings and demonstrate the broad range of applications of our method by applying it to problem instances of least squares hyperparameter auto-tuning, multi-class SVM auto-tuning, and inverse optimal control.

\end{abstract}

\vspace*{-0.1cm}
\section{INTRODUCTION}\label{sec:introduction}

\begin{figure}[!t]
\centering
\includegraphics[width=0.95\linewidth]{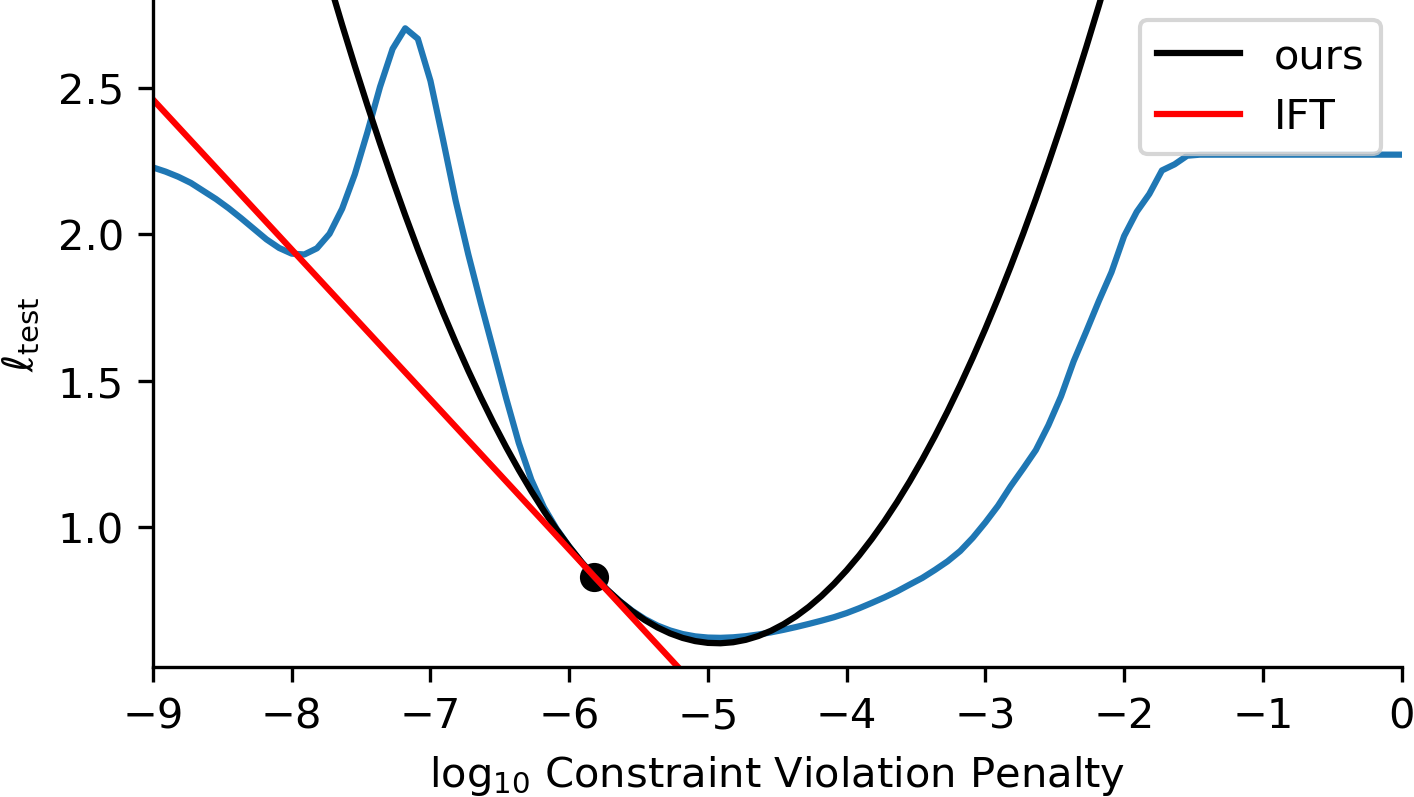}
\caption{The single \mbox{(hyper-)parameter} test loss landscape of a multi-class SVM on Fashion-MNIST. Evaluating a point on this curve takes $\sim$100 seconds. We obtain a local quadratic approximation which leads to a much faster \mbox{(hyper-)parameter} optimization.}
\label{fig:abstract}
\vspace*{-0.4cm}
\end{figure}

Optimization is the foundation of modern learning and decision making systems, therefore a natural problem of interest is how to improve, learn, or optimize the optimization itself.
Many practitioners of autonomous driving, robotics, and machine learning employ optimization on an everyday basis. Understanding how best to adjust this tool to more accurately suit their application needs is key to improving performance, trust, and reliability of these systems.

A natural way to approach improving optimization comes through formulating a bilevel program---users solve an optimization problem constructed with given data and parameters, and rely on a secondary metric quantifying the quality of the optimization result (it is of course optimal with respect to its own objective) to inform system design. Examples include (i) tuning regularization of a regression model to give good results on the test set, while training it to optimality on the train set, (ii) designing an autonomous car that drives in a human-like fashion, where it optimizes a finite horizon trajectory planning problem at every time step, (iii) setting parallel auction prices in such a way that rational bidding (an optimization in itself) leads to highest auction holder revenue. Any time optimization or decision making is applied, the question of selecting the right \mbox{(hyper-)parameters} arises in order to obtain, for example, (i) statistical models which generalize better, (ii) autonomous agents that behave more like an expert, (iii) auction systems that cannot be exploited. This notion is formalized as a bilevel program in which the \emph{optimization-to-be-improved} represents the lower level:
\begin{equation}\label{eq:bilevel}
\begin{aligned}
\minimize_p ~~~ & f_U(\zs, p) \\
\st ~~~ & \zs = \argmin_{z \in \mathbb{Z}} f_L(z, p).
\end{aligned}
\end{equation}
This bilevel program formulation is general and subsumes the problems of test set model generalization, Stackelberg competition \citep{von2010market}, meta-learning \citep{finn2017model} and few-shot learning \citep{lee2019meta}.
The quality of the optimization result, $\zs$, of the objective $f_L$ is quantified via the upper objective, $f_U$. In the example of a statistical model, $f_L$ represents the loss on the train dataset and $f_U$ the loss on the test dataset. Solving the bilevel program requires selecting parameters $p$ which produce such $\zs$ that together lead to the minimal upper level loss $f_U$.\footnote{When $\zs$ itself has an interpretation as ``parameters'', e.g., in model learning, $p$ may be referred to as ``hyperparameters''; we will refer to $p$ as parameters throughout the remainder of this domain-agnostic work.}

This solution is often approximated by selecting parameters $p$ by hand or via grid search. However, these approaches suffer from (a) being limited to cases where the dimension of $p$ is low (usually below 4), (b) requiring parallel computing resources to keep re-evaluating $\zs(p)$ and most critically, (c) the search is often done manually and wastes the expert's or practitioner's time.

A more principled way of solving Problem~\eqref{eq:bilevel} is to use derivative information and make use of general-purpose solvers developed for optimization problems. This, however, requires the derivative of $\zs$ with respect to $p$---quantifying how small changes in $p$ affect the upper level objective $f_U$ not just directly, but also by influencing $\zs$. Although a closed-form expression of $\zs$ with respect to $p$ rarely exists, because $\zs$ is the result of optimization, the dependence is \emph{implicitly} defined via the necessary conditions for optimality of the lower optimization. For smooth-in-parameter problems, the implicit function theorem (IFT) may be employed to compute this derivative information relevant to solving the upper problem.

The availability of gradient expressions in bilevel programming obtained via sensitivity analysis enables the use of existing powerful optimizer to tackle these problems when they arise in real-world applications. However, even though many optimization algorithms use second-order information to converge fast in cases where the forward function evaluation is the bottleneck, not much attention has been paid to extending sensitivity analysis to second-order information. Doing so would enable another class of faster optimization algorithms to be applied to bilevel programming.

The theoretical application of sensitivity analysis to bilevel programming relies on exact solutions to the lower level problems, but in reality, the numerical limitations rarely allow for that. Existing literature on first-order methods thus focuses on showing that the error in the derivative can be bounded and goes to zero as the approximation approaches the solution to the lower level problem.

\subsection{Contributions}

In this work we extend the application of the implicit function theorem, where \emph{gradients} of inner optimization result with respect to parameters are found as in many existing works, \citep{gould2016differentiating, agrawal2019cone, barratt2018differentiability}, and derive second-order derivatives, i.e., the IFT Hessian. We leverage this result for three main contributions: (i) We show that the computational complexity of obtaining second-order derivatives is, in many cases, still dominated by the same matrix inversion bottleneck required for the IFT gradient and so our method can be implemented equally efficiently. (ii) We analyze our IFT Hessian expression to derive computational complexity and error bound expressions. We derive a new form of the regularized error bound under diagonal regularization of the matrix inverse operation in the application of IFT. (iii) We use our second-order derivative expression to apply second-order optimization methods to two machine learning datasets and show that these methods lead to faster bilevel optimization, requiring fewer lower level problem evaluations. 

We then further discuss the practical limitations and advantages of second-order optimization for bilevel optimization.

We open-source our implementation in two popular machine learning/scientific computing/automatic differentiation frameworks, PyTorch\footnote{\href{https://pytorch.org/}{pytorch.org}} and JAX\footnote{\href{https://github.com/google/jax}{github.com/google/jax}}, in a user-friendly format at 
\url{https://github.com/StanfordASL/sensitivity\_torch}
and 
\url{https://github.com/StanfordASL/sensitivity\_jax}.

\subsection{Related Work}\label{sec:related_work}

\paragraph{Practical Deployment} Optimization improvement or optimization tuning has a long history in practical applications. For systems for which gradient derivation is non-trivial or more generally for systems where local gradient information is not informative of the global scope of the problem, gradient-free proxy models may be employed as in \citet{golovin2017google}. Like hand-tuning or grid search, this approach constrains the number of parameters that can practically be tuned to single or low double digits.

\paragraph{Formal Literature} More formally the Bilevel Programming Problem (BLPP) formulation has a long history in literature \citep{bard1984optimality, von2010market, henrion2011calmness, liu2001exact, bard2013practical}. \citet{sinha2017review} contains an extensive review of approaches to solving BLPPs.
Many of these works focus on theoretical analysis/characterization of BLPP approaches; the focus of this paper is more on specific concrete applications.

\paragraph{Applications-oriented} Renewed interest in applications-oriented gradient-based solutions to BLPPs \citep{bengio2000gradient}, led to several works establishing the techniques for obtaining lower level solution gradients with respect to the parameters \citep{gould2016differentiating} and doing so efficiently for convex problems \citep{barratt2018differentiability, agrawal2019cone, agrawal2019layers}. Several computationally optimized, program-form specific approaches have been shown \citep{amos2017optnet, amos2018differentiable}. Most recently \citet{lorraine2020optimizing, blondel2021efficient} apply the techniques to large-scale programs. Most of these applications-oriented works focus on deriving gradient expressions---to be used with a gradient-only BLPP optimizer. These works do not consider the loss landscape or the local curvature of the BLPP; in contrast, in this work, we attempt to quantify that.

\paragraph{Machine Learning} BLPPs also found applications in meta-learning literature \citep{andrychowicz2016learning, finn2017model, harrison2018meta, bertinetto2018meta} with several works making explicit use of the implicit function theorem \citep{lee2019meta, rajeswaran2019meta}. While meta-learning literature poses an important application for BLPP, so far little attention has been given to improving the specifics of the solution methods employed in this body of work.

\paragraph{Higher-Order Derivatives} Works most closely related to ours, with a focus on finding higher derivative information and analyzing the curvature of the BLPP are \citet{wachsmuth2014differentiability, mehlitz2021sufficient}. These works do not demonstrate how to efficiently compute second order derivatives or demonstrate their usage in Newton's-Method-like optimization, two key focuses of our work.

\subsection{Notation} 
For an argument $x \in \mathbb{R}^{d}$ we denote the dimension of $x$ by $\dim(x) = d$. For a scalar function $f : \mathbb{R}^{\dim(x)} \rightarrow \mathbb{R}$ we denote its gradient and Hessian as $\nabla_x f(x)$ and $\nabla^2_x f(x)$. For a vector function of $g : \mathbb{R}^{\dim(x)} \rightarrow \mathbb{R}^{\dim(g)}$ we denote its Jacobian matrix as $D_x g(x) \in \mathbb{R}^{\dim(g) \times \dim(x)}$. Where a function takes two arguments, we use the normal and mono-space font respectively to denote whether the differentiation operator is partial or total (i.e. does the derivative capture dependence between variables), e.g. for a function $g : \mathbb{R}^{\dim(x)} \times \mathbb{R}^{\dim(y)} \rightarrow \mathbb{R}^{\dim(g)}$, $D_x g(x, y)$ denotes the partial Jacobian of $g$ w.r.t. $x$ and $\mathtt{D}_x g(x)$ the total Jacobian of $g$ w.r.t. $x$ accounting for possible dependence of $x$ on $y$. In this work, we aim to describe higher order derivatives of vector functions, which leads us to define, for $g : \mathbb{R}^{\dim(x)} \times \mathbb{R}^{\dim(y)} \rightarrow \mathbb{R}^{\dim(g)}$, $D_{xy} g \in \mathbb{R}^{\dim(g)\dim(x) \times \dim(y)}$, represented as a two dimensional matrix s.t. 
\begin{equation}
D_{x y} g(x, y) = \begin{bmatrix}
D_y \big(\nabla_x \left( g(x, y)_1 \right) \big) \\
\vdots \\
D_y \big(\nabla_x \left(g(x, y)_{\dim(g)} \right) \big) \\
\end{bmatrix}
\end{equation}
where $g(x, y)_i$ denotes the $i$-th scalar output of the vector function $g$. We define the operator $H_x \equiv D_{x x}$ (and the total version $\mathtt{H}_x \equiv \mathtt{D}_{x x}$). We use $\otimes$ to denote the Kronecker product.

\section{PROBLEM STATEMENT}\label{sec:problem}

Recall the formulation of bilevel programming:
\begin{equation}
\begin{aligned}
\minimize_p ~~~ & f_U(\zs, p) \\
\st ~~~ & \zs = \argmin_{z \in \mathbb{Z}} f_L(z, p).
\end{aligned}\tag{1}
\end{equation}
The problem considered in this work is how to obtain an explicit expression for the total second-order derivative, the Hessian matrix, of $f_U$ with respect to $p$, i.e., $\mathtt{H}_p f_U(\zs, p)$. Access to this Hessian enables the application of second-order optimization methods for solving the upper problem, with the aim of reducing the total number of lower problem optimizations (and correspondingly, total overall computation).

\section{METHOD}\label{sec:method}

\subsection{Preliminaries}

\paragraph{Necessary Derivatives}
As overviewed in Section~\ref{sec:related_work}, there are several approaches to solving the bilevel problem~\eqref{eq:bilevel}. Here, we focus on derivation of the first and second derivatives of $f_U$ w.r.t. $p$: $\mathtt{D}_p f_U$ and $\mathtt{H}_p f_U$. Because $\zs$ depends on $p$, the total derivatives of $f_U$ can be written as
\begin{align}
\label{eq:chain_1st}
\mathtt{D}_p f_U(\zs, p) =& ~ D_p f_U + \big(D_\zs f_U \big) \big(D_p \zs \big) \\
\mathtt{H}_p f_U(\zs, p) =& ~ H_p f_U + \big(D_p \zs)^T \big(H_\zs f_U\big) \big(D_p \zs) \nonumber \\
\label{eq:chain_2nd}
& + \big(\big(D_\zs f_U\big) \otimes I\big) H_p \zs 
\end{align}
Importantly, Equations \eqref{eq:chain_1st} and \eqref{eq:chain_2nd} depend on (i) terms \emph{directly obtainable} from the upper objective function via analytical or automatic differentiation: $D_p f_U$, $H_p f_U$, $D_\zs f_U$, $H_\zs f_U$ and (ii) terms quantifying the sensitivity of lower level optimization, i.e., the result $\zs$ w.r.t. $p$: $D_p \zs$, $H_p \zs$.

The latter two terms are nominally obtainable through automatic differentiation by unrolling the entire $f_L$ optimization process---in many cases this is computationally undesirable or simply too memory intensive to be feasible. An alternative way of obtaining optimization sensitivity terms follows from the IFT.

\paragraph{Implicit Function} The main insight which allows differentiating $\zs = \argmin_z f_L(z, p)$ is the implicit condition imposed on $\zs$ as a result of $\zs$ being an optimal solution. An optimal solution must satisfy First Order Optimality Conditions (FOOC)
\begin{equation}
k(\zs, p) = 0
\end{equation}
which define an implicit equation for $\zs$.\footnote{An \emph{implicit function} is simply defined as a function $k : \mathbb{R}^m \times \mathbb{R}^n \rightarrow \mathbb{R}^m$ that equals 0 if its first input $x$ satisfies the implicit definition: $x \in \{x \mid k(x, y) = 0, ~ \exists y \in \mathbb{R}^n\}$. This is in contrast to an \emph{explicit function} of the form $f : \mathbb{R}^n \rightarrow \mathbb{R}^m$: $x \in \{x \mid f(y) = x\}$.} For unconstrained optimization, the FOOC we consider are explicitly:
\begin{equation*}
k(\zs, p) = D_\zs f_L(\zs, p) = 0.
\end{equation*}

The following first-order sensitivity analysis is standard in the literature; we reproduce it here to motivate and enable our second-order analysis.

\subsection{Implicit Function Theorem (IFT)}

\begin{thm}[Implicit Function Theorem (IFT)]
Let $k: \mathbb{R}^{m} \times \mathbb{R}^{n} \rightarrow \mathbb{R}^m$ be a continuously differentiable multivariate function of two variables $\zs \in \mathbb{R}^m$ and $p \in \mathbb{R}^n$ such that $k$ defines a fixed point for $\zs$, i.e., $k(\zs, p) = 0$ for all values of $p \in \mathbb{R}^n$. Then, the derivative of (implicitly defined) $\zs$ w.r.t. $p$ is given by
\begin{equation*}
D_p \zs = -\big( D_\zs k(\zs, p) \big)^{-1} D_p k(\zs, p).
\end{equation*}
\end{thm}

\begin{proof}
Take $k(\zs, p) = 0$ and apply the chain rule to differentiate w.r.t. $p$, then, if $D_\zs k(\zs, p)$ is invertible
\begin{align}
k(\zs, p) &= 0\nonumber \\ 
D_p k(\zs, p) + \big( D_\zs k(\zs, p) \big) \big( D_p \zs \big) &= 0 \nonumber \\
-\big( D_\zs k(\zs, p) \big)^{-1} \big( D_p k(\zs, p) \big) &= D_p \zs.
\label{eq:diff_1st}
\end{align}
\end{proof}

Understanding the proof of the equation above is vital to obtaining higher derivatives via the IFT, since under suitable smoothness assumptions, Equation~\eqref{eq:diff_1st} can be differentiated again and the resulting expression solved for $H_p \zs$.

We now present the rarely derived IFT Hessian, which allows us to obtain the second derivative through an optimization. This result follows from a repeated application of the IFT to the same implicit function to obtain higher derivatives of $\zs$ w.r.t. $p$.

\subsection{Second-Order IFT}

\begin{thm}[Second-Order IFT]\label{thm:hessian}
Let $k: \mathbb{R}^{m} \times \mathbb{R}^{n} \rightarrow \mathbb{R}^m$ be a twice continuously differentiable multivariate function of two variables $\zs \in \mathbb{R}^m$ and $p \in \mathbb{R}^n$ such that $k$ defines a fixed point for $x$, i.e. $k(\zs, p) = 0$ for all values of $p \in \mathbb{R}^n$. Then, the Hessian of (implicitly defined) $\zs$ w.r.t. $p$ is given by
\begin{equation}\label{eq:2nd_hess}
\begin{aligned}
H_p \zs = -\bigg[ \big(D_\zs k \big)^{-1} \otimes I \bigg] \bigg[ 
H_{p} k + \big( D_{p \zs} k \big) \big( D_p \zs \big) \\
~~~~~~~~~~~~+ \big( I \otimes \big(D_p \zs\big)^T \big) \big( D_{\zs p} k \big) \\
~~~~~~~~~~~~+ \big( I \otimes \big(D_p \zs\big)^T \big) \big( H_{\zs} k \big) \big( D_p \zs \big) \bigg].
\end{aligned}\tag{7}
\end{equation}
\end{thm}

The proof of Theorem~\ref{thm:hessian} is provided in the \hyperref[sec:appendix]{Appendix}.

\section{ANALYSIS}\label{sec:analysis}

\subsection{Computational Complexity}\label{sec:comp_complex}
The expression for the Second-Order IFT features Kronecker product terms (denoted by $\otimes$) which serve to broadcast over dimensions of either the embedding $\dim(z)$ or the parameter $\dim(p)$. Efficiently broadcasting over a particular dimension does not require constructing full dense matrices. 

The two broadcasting Kronecker product-based forms that appear in Equation~\eqref{eq:chain_2nd} are of the form $A \otimes I$ and $I \otimes B$ which both correspond to a broadcasted version of matrix multiplication. We refer the interested reader to the \hyperref[sec:appendix]{Appendix} for the discussion of how computation with these forms can be accomplished efficiently.

Since our analysis here focuses on sensitivity analysis for optimization, we devote most of our attention to analyzing the computational complexity of computing the expression in Equation~\eqref{eq:chain_2nd}, which we recall here in full
\begin{align*}
\mathtt{H}_p f_U(\zs, p) =& ~ H_p f_U + \big(D_p \zs)^T \big(H_\zs f_U\big) \big(D_p \zs) \\
& + \big(\big(D_\zs f_U\big) \otimes I\big) H_p \zs 
\end{align*}
where $H_p f_U$, $H_\zs f_U$ and $D_\zs f_U$ have explicit expressions and the term $D_p \zs$ is either already available to us from first-order analysis or can be computed cheaply since the matrix $D_\zs k$ had to be factorized for first-order analysis. Thus, the only term requiring significant computation, substituting Equation~\eqref{eq:2nd_hess}, is
\begin{align*}
\big(\big(D_\zs f_U\big) \otimes I\big) \underbrace{\bigg[ \big(D_\zs k \big)^{-1} \otimes I \bigg] \bigg[ \dots \bigg]}_\text{Equation~\eqref{eq:2nd_hess}} = \\
\bigg[ \left( D_\zs f_U \big(D_\zs k \big)^{-1}\right) \otimes I \bigg] \bigg[ \dots \bigg]
\end{align*}

Since the upper level objective is a scalar by definition, the product 
\begin{equation}
\label{eq:sens_vec}
D_\zs f_U \big(D_\zs k \big)^{-1} = v^T \in \mathbb{R}^{1 \times m}
\end{equation}
can be computed at the cost of a single matrix solve using an already necessarily factorized matrix from first-order analysis, \emph{which can be done computationally cheaply.} Thus, evaluating the term $\big(\big(D_\zs f_U\big) \otimes I\big) H_p \zs$ reduces to (a) caching a term from first-order analysis and (b) broadcasted vector-matrix product (a weighted summation of matrices). Alternatively, if an automatic differentiation system is used to compute the right bracket in Equation~\eqref{eq:2nd_hess}, then $v$ can be used as a sensitivity vector in vector-Jacobian or Jacobian-vector products, significantly reducing the number of calls to the automatic differentiation (autodiff) engine.

\subsubsection{Big-\Oh ~Notation}

{
\renewcommand{\arraystretch}{1.1}
\begin{table}[h]
\begin{tabular}{l | l}
Operation & Computational Complexity \\ \hline
1st IFT & \Oh$\left(m^3 + m n \right)$ \\
1st IFT w/ sens. & \Oh$\left(m^3 + n\right)$ \\
2nd IFT & \Oh$\left(m^3 + m^2 n^2 \right)$ \\
2nd IFT w/ sens. & \Oh$\left(m^3 + m n^2 \right)$
\end{tabular}
\caption{Computational complexity of applying the first- and second-order implicit function theorem, omitting computation of partial explicit derivatives.
In this table we define $m = \dim(z)$, $n = \dim(p)$. ``w/ sens.'' denotes that a sensitivity vector is available in which case the matrix inverse can be computed for a single left hand side as in Equation~\eqref{eq:sens_vec}.}
\label{tab:comp_complex}
\end{table}
}
Following Section~\ref{sec:comp_complex} we show computational complexity of applying the first- (1st) and second-order (2nd) implicit function theorem in Table \ref{tab:comp_complex} and observe that the second-order expression with a sensitivity left-hand side---the term $D_\zs f_U$ in Equations \eqref{eq:chain_1st} \eqref{eq:chain_2nd}---has a computational complexity that differs from first-order expression with a sensitivity left hand side only by the additional quadratic term in $n = \dim(p)$, which is expected as the resulting matrix is in $\mathbb{R}^{n \times n}$. The increase in computational complexity is thus minor. Here we use the computational complexity of a matrix factorization operation to be \Oh$\left(m^3\right)$.

We do not include the computational complexity of obtaining the partial derivatives of $f_U$ and $k$, which we denote as \emph{partial explicit derivatives}, primarily because: they are heavily problem dependent; they vanish; can be precomputed; efficient analytical expressions exist; or they can be obtained by efficiently making use of Jacobian-vector and vector-Jacobian products in an automatic differentiation engine.

\subsection{Error Analysis}

Error analysis is vital for sensitivity analysis because numerical limitations often result in lower level solutions that are not quite optimal, but are instead at a small distance from the optimum. Developing confidence in the sensitivity methods described here requires the ability to bound the error caused by applying sensitivity analysis developed for optimal points to suboptimal lower level solutions.

\begin{figure*}[!t]
\centering
\subfigure[\texttt{RR}]{
\includegraphics[width=0.31\linewidth]{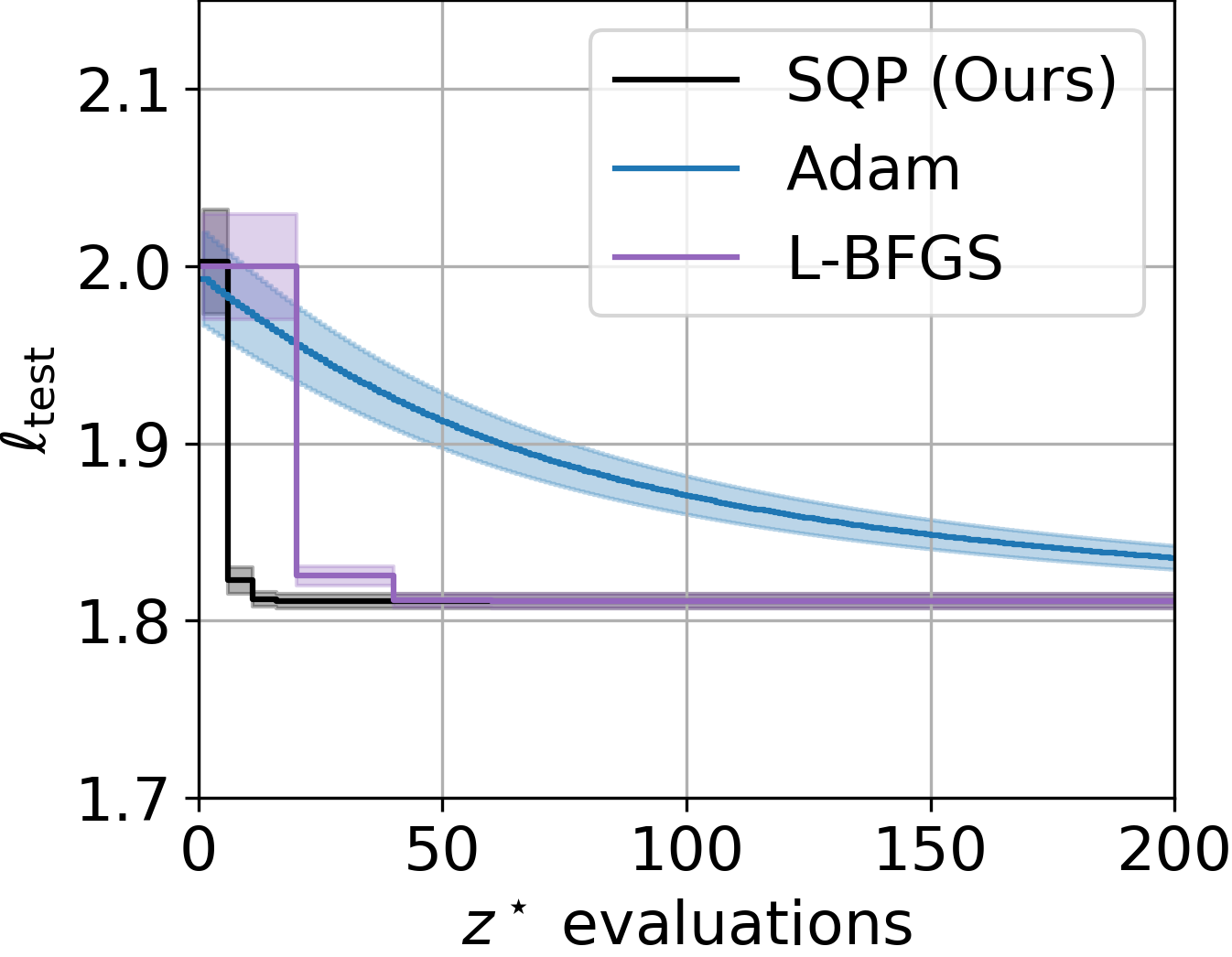}
}
\subfigure[\texttt{diag}]{
\includegraphics[width=0.31\linewidth]{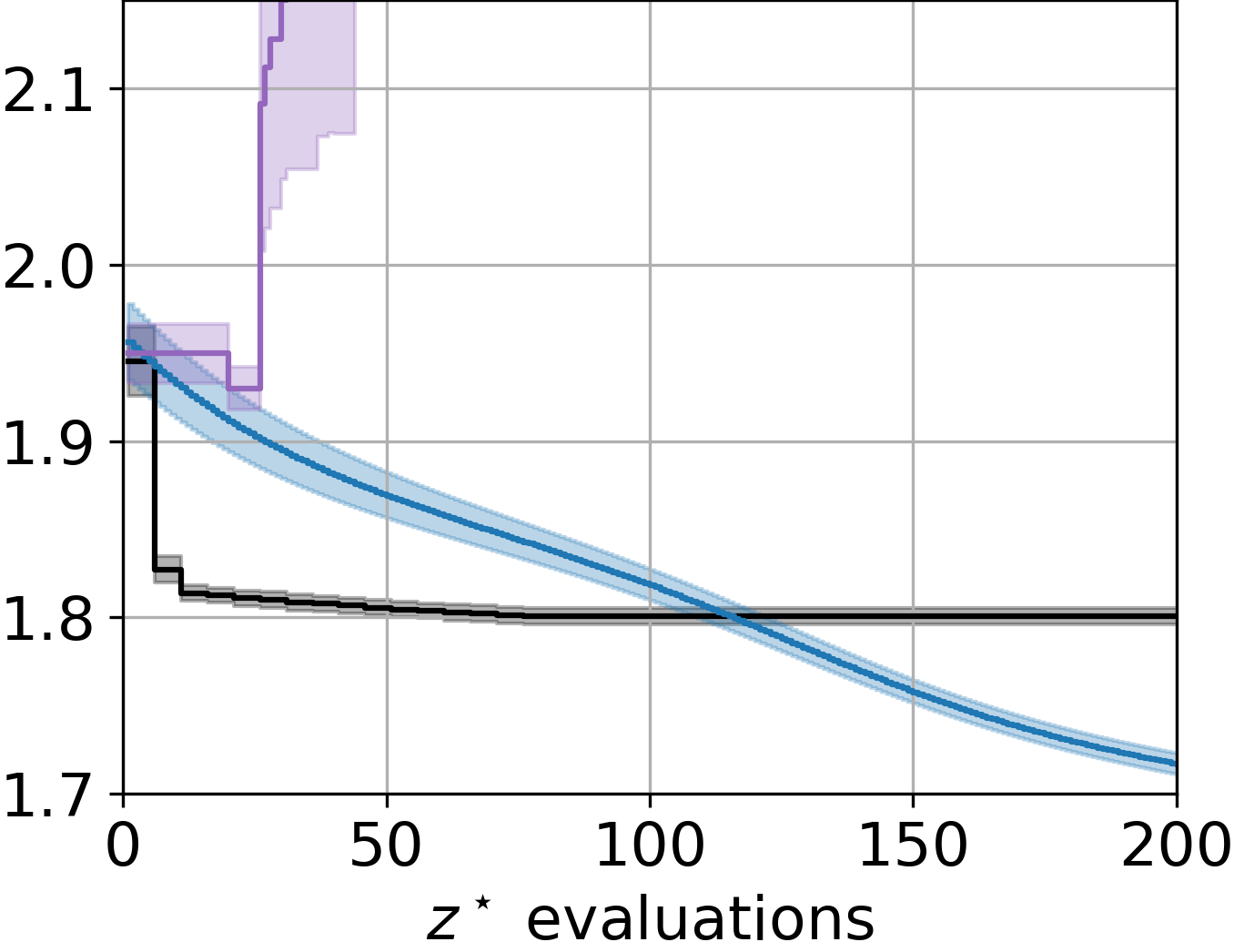}
}
\subfigure[\texttt{conv}]{
\includegraphics[width=0.31\linewidth]{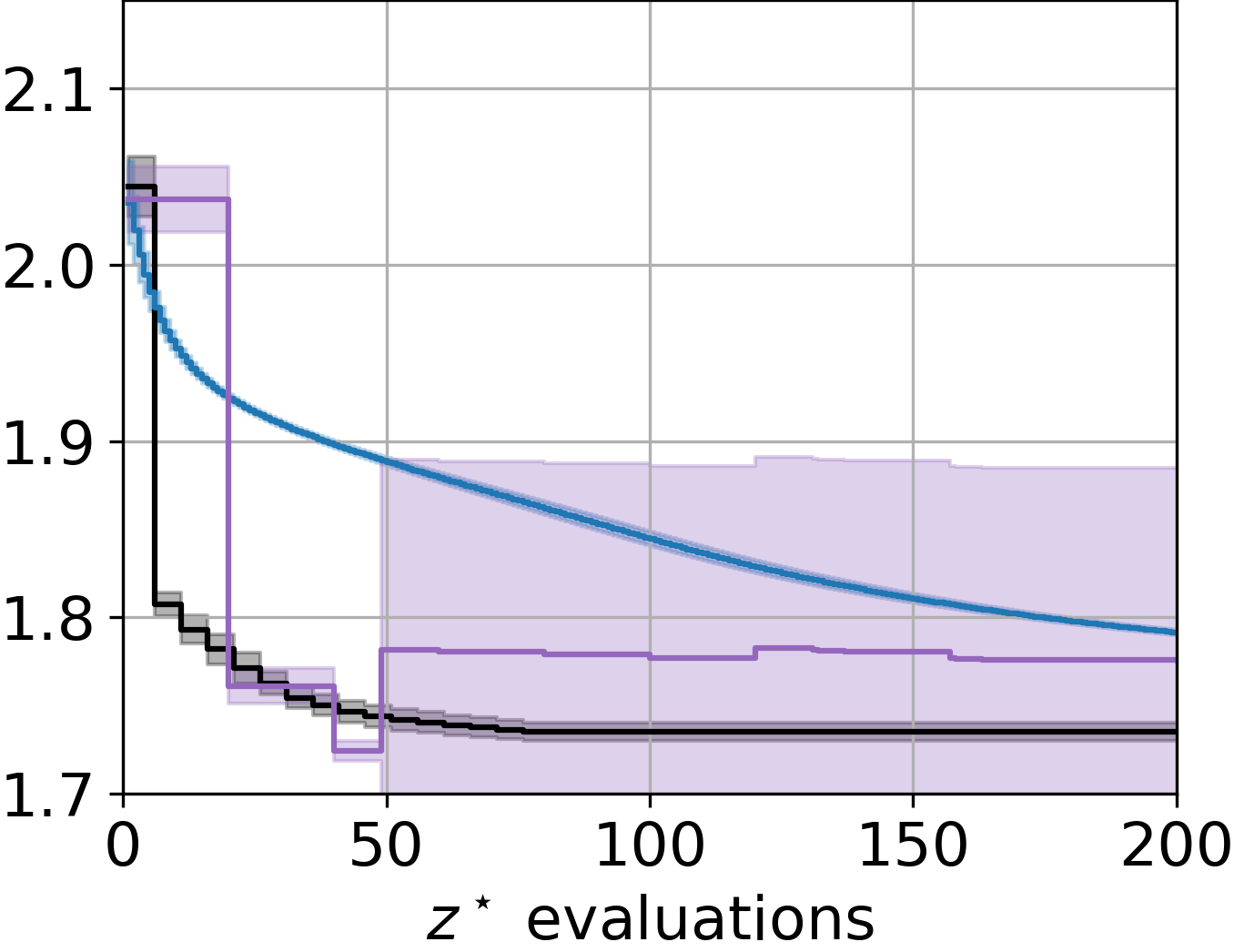}
}
\vspace*{-0.3cm}
\caption{Hyperparameter optimization of least-squares models with two gradient-only algorithms and one gradient \& Hessian---enabled by this work.}
\label{fig:linear_models_loss}
\end{figure*}

\begin{figure*}[!t]
\centering
\subfigure[\texttt{RR}]{
\includegraphics[width=0.31\linewidth]{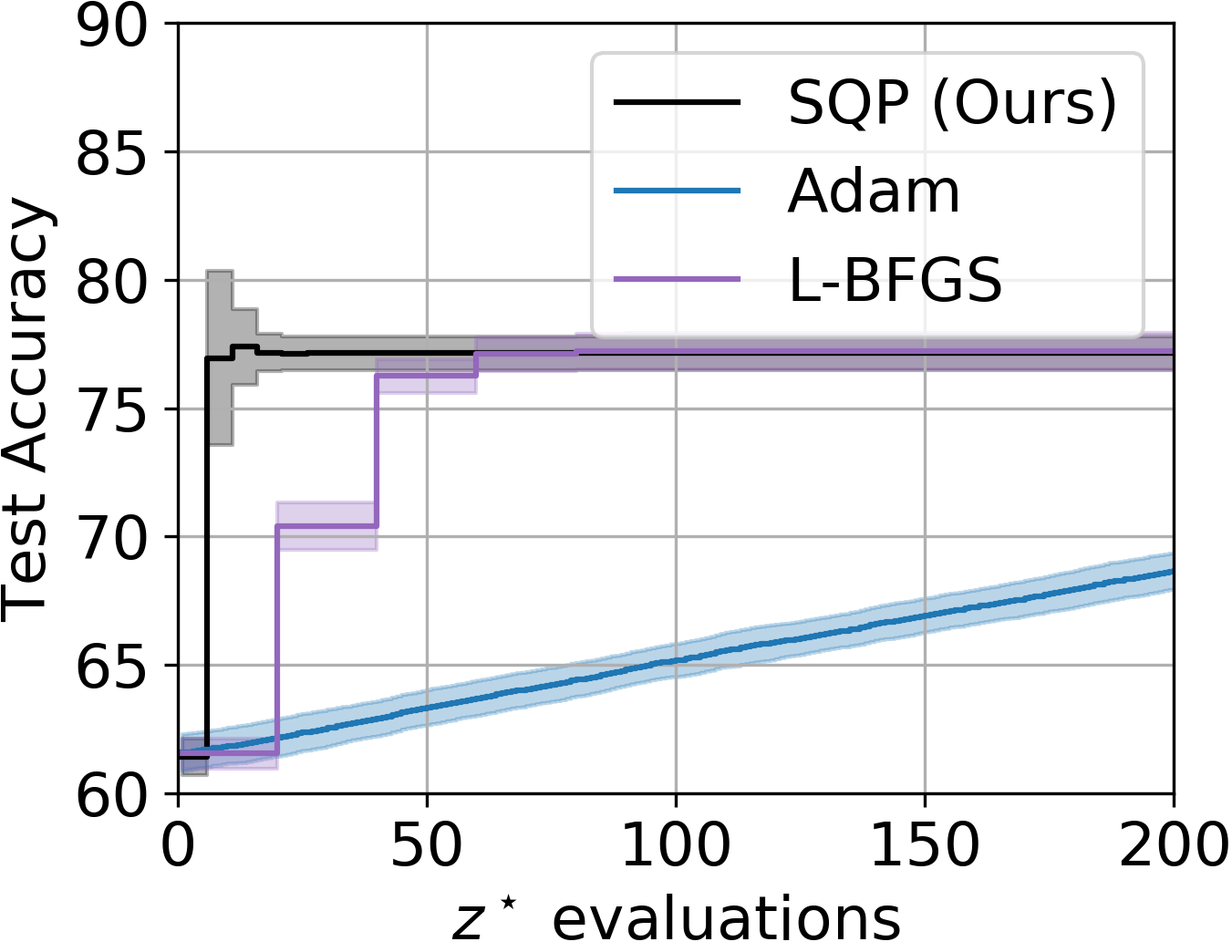}
}
\subfigure[\texttt{diag}]{
\includegraphics[width=0.31\linewidth]{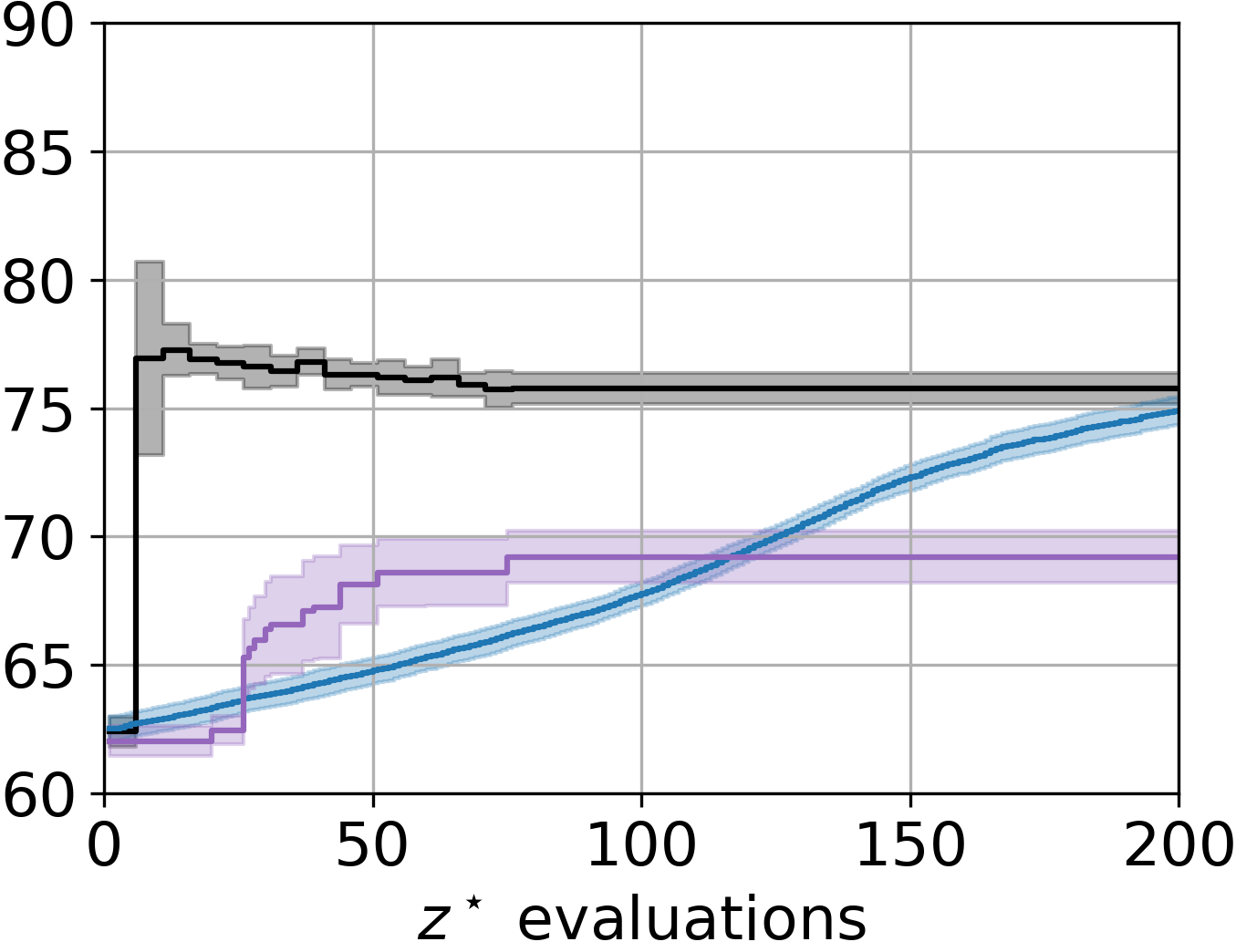}
}
\subfigure[\texttt{conv}]{
\includegraphics[width=0.31\linewidth]{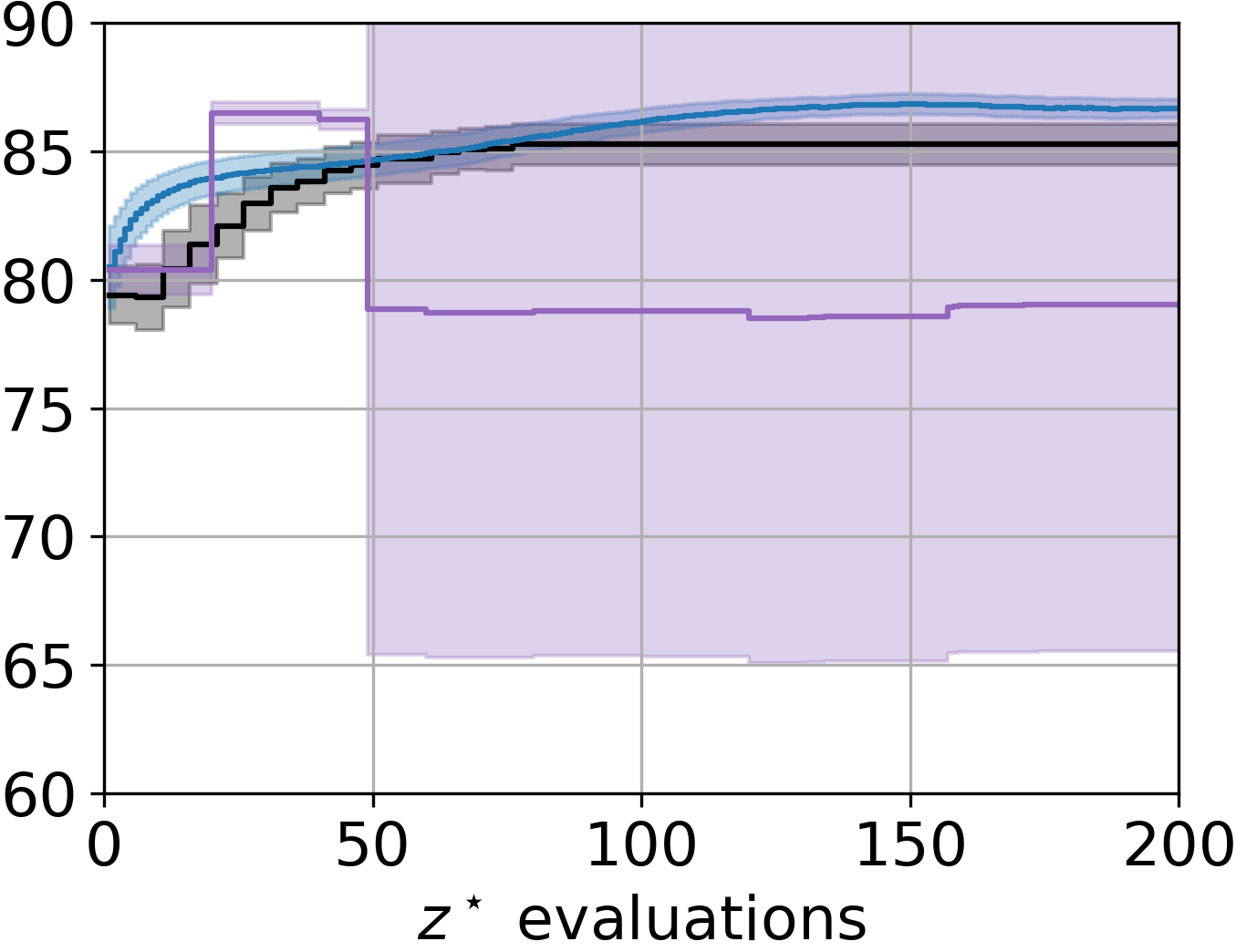}
}
\vspace*{-0.3cm}
\caption{Hyperparameter optimization of least-squares models with two gradient-only algorithms and one gradient \& Hessian---enabled by this work.}
\label{fig:linear_models_acc}
\end{figure*}


\subsubsection{First-Order Error Bound}

Assuming an inexact local solution, we state the following error bound on the Jacobian. Theorem 3 is heavily inspired by \citet{blondel2021efficient}, Theorem 1, which is in turn inspired by \citet{higham2002accuracy}, Theorem 7.2.
\begin{thm}[First-Order Error Bound]\label{thm:error1}
Given the result of IFT applied to an exact solution $D_\zs p = g = -A^{-1} B$, where $A(\zs, p) = D_\zs k(\zs, p)$, $B(\zs, p) = D_p k(\zs, p)$ and an inexact solution $\widetilde{A}(z, p) = D_z k(z, p)$ (where $k(z, p) \neq 0$) and $\widetilde{B}(z, p) = D_p k(z, p)$. Assume $\norm{z - \zs} \leq \delta$,
$\norm{\widetilde{A} - A}_\text{op} \leq \gamma \delta$,
$\norm{\widetilde{B} - B}_F \leq \beta \delta$,
$\norm{B}_F \leq R$,
$\norm{\widetilde{A} v} \geq \alpha_1 \norm{v}$, $\norm{A v} \geq \alpha_2 \norm{v}$,
then
\begin{equation}
\norm{\widetilde{J} - J}_F \leq \frac{\beta}{\alpha_1} \delta + \frac{\gamma R}{\alpha_1 \alpha_2} \delta
\end{equation}
\end{thm}

We point the reader to the \hyperref[sec:appendix]{Appendix} for the proof.

\subsubsection{Second-Order Error Bound}

Assuming an inexact local solution, we state the following error bound on the Jacobian.
\begin{thm}[Second-Order Error Bound]\label{thm:error2}
Given the result of IFT applied to an exact solution, for $\zs \in \mathbb{R}^m$, $H_\zs p = H = -A^{-1} B$, where $A(\zs, p) = D_\zs k(\zs, p) \otimes I$, $B(\zs, p) = H_{p} k + \big( D_{p \zs} k \big) \big( D_p \zs \big)
+ \big( I \otimes \big(D_p \zs\big)^T \big) \big( D_{\zs p} k \big) 
+ \big( I \otimes \big(D_p \zs\big)^T \big) \big( H_{\zs} k \big) \big( D_p \zs \big)$ and an inexact solution $\widetilde{A}(z, p) \neq 0$ and $\widetilde{B}(z, p)$. Assume $\norm{z - \zs} \leq \delta$, 
$\norm{\widetilde{A} - A}_\text{op} \leq \gamma \delta$,
$\norm{H_p k(z, p) - H_p k(\zs, p)}_F \leq \zeta \delta$,
$\norm{D_{z p} k(z, p) - D_{\zs p} k(\zs, p)}_F \leq \eta \delta$,
$\norm{H_z k(z, p) - H_\zs k(\zs, p)}_F \leq \nu \delta$,
$\norm{D_z p - D_\zs p}_F \leq \kappa_J \delta$ defined in Theorem \ref{thm:error1},
$\norm{\widetilde{A} v} \geq \alpha_1 \norm{v}$, $\norm{A v} \geq \alpha_2 \norm{v}$,
then
\begin{equation}
\norm{\widetilde{H} - H}_F \leq \frac{\zeta + 2 \eta \kappa_g + \nu \kappa_g^2}{\alpha_1} \delta + \frac{\gamma R_H}{\alpha_1 \alpha_2} \delta.
\end{equation}
\end{thm}

We point the reader to the \hyperref[sec:appendix]{Appendix} for the proof.

\subsubsection{Error Bound Optimization}\label{sec:error_bound_opt}
Now, assuming an inexact solution we analyze the diagonal regularization of the inverse in the application of IFT. We let $\widehat{A} = \widetilde{A} + \epsilon I$.

\begin{thm}[Regularized First-Order Error Bound]\label{thm:error1_reg}
Given the result of regularized IFT applied to an exact solution $D_\zs p = g = -A^{-1} B$, where $A(\zs, p) = D_\zs k(\zs, p)$, $B(\zs, p) = D_p k(\zs, p)$ and an inexact solution $\widetilde{A}(z, p) = D_z k(z, p) + \epsilon I$, $\widetilde{B}(z, p) = D_p k(z, p)$ (where $k(z, p) \neq 0$) . Assume
$\norm{z - \zs} \leq \delta$,
$\norm{\widetilde{A} - A}_\text{op} \leq \gamma \delta$,
$\norm{\widetilde{B} - B}_F \leq \beta \delta$,
$\norm{B}_F \leq R$,
$\norm{\widetilde{A} v} \geq \alpha_1 \norm{v}$ and $v^T \widetilde{A} v \geq 0$, so
$\norm{\widehat{A} v} \geq (\alpha_1 + \epsilon) \norm{v}$,
$\norm{A v} \geq \alpha_2 \norm{v}$
then
\begin{equation}
\norm{\widehat{J} - J}_F \leq \frac{\beta \delta}{\alpha_1 + \epsilon} + \frac{R\big(\gamma \delta + \epsilon\big)}{\big(\alpha_1 + \epsilon\big) \alpha_2}
\end{equation}
\end{thm}

Selecting a \emph{post hoc} arbitrary regularization allows to tighten the bound. We refer the reader to the \hyperref[sec:appendix]{Appendix} for the proof.

\section{EXPERIMENTS}\label{sec:experiments}

\subsection{Regression Model Auto-Tuning}

We compare the performance of three commonly used nonlinear optimization algorithms on the problem of linear model improvement via smooth hyperparameter tuning. Using the above analysis we are able to apply a second-order optimization method, which offers to dramatically reduce the number of lower function evaluations in BLPP and significantly speed up optimization problems where the lower level constraint evaluation dominates. To compare, we optimize the hyperparameters of 3 linear models on MNIST \citep{lecun1998mnist} using three commonly used optimization algorithms, two gradient-only: (i) Adam \citep{kingma2014adam}, (ii) L-BFGS \citep{liu1989limited} and one using second-order information (enabled by this work) (iii) SQP \citep{nocedal2006numerical}. The three linear models we choose all employ a least-squares lower level loss where the target vector is the one-hot encoding of the ten MNIST digits. We train on 1000 randomly selected MNIST examples in the train set and evaluate the upper level (test) loss on all examples in the test set. The upper level loss is the cross-entropy classification loss, $f_U(\zs, p) = -\sum_j^{N_\text{test}}\sum_{i=1}^{10} t_{i,j} \log\big(q_{i,j}\big)$ where $t_{i,j} = \delta_{i,y_j}$ and $q_{i,j} = e^{x_j^T \zs_i} / \left( \sum_{i=1}^{10} e^{x_j^T \zs_i} \right)$.\footnote{$t_{i,j} = \delta_{i,y_j}$ indicates that $t_{i,j}$ is equal to 1 if example $j$ is of the class $i$ and 0 otherwise.} For all models $z \in \mathbb{R}^{n d}$ where $n = 10$ is the number of MNIST classes and $d$ is the number of features in the data vector.

\paragraph{Model 1 (\texttt{RR})} A single hyperparameter least-squares model with Tikhonov regularization \citep{tikhonov1943stability}, also known as ridge regression \citep{gruber2017improving}. The lower level loss takes the form $f_L(z, p) = \norm{X z - Y}_\text{F}^2 + 10^p \norm{z}_\text{F}^2$. The features are raw image pixels and a bias term.
\paragraph{Model 2 (\texttt{diag})} A least-squares model with Tikhonov regularization where each weight is penalized with a separate weight. The lower level loss takes the form $f_L(z, p) = \norm{X z - Y}_\text{F}^2 + \sum_i^{\dim(z)} 10^{p_i} z_i^2$. The features are the same as in Model 1.
\paragraph{Model 3 (\texttt{conv})} A least-squares model where the images are first passed through a parametric 2D convolution filter with a $3 \times 3$ kernel, a stride of 2, a bias, 1 input channel and 2 output channels. The convolution output is passed through the $\tanh$ activation function, before adding bias and applying the Tikhonov regularized least-squares model. The reduced images have a dimension of 338. The convolution weight, bias and the scalar Tikhonov regularization weight form the vector $p$.

We show the optimization results in Figures \ref{fig:linear_models_loss} \& \ref{fig:linear_models_acc}. The use of second-order derivative information significantly reduces the number of least-squares evaluations.

In \texttt{RR} our method, SQP, outperforms other optimizers both in terms of test accuracy and test loss. In \texttt{diag} L-BFGS goes unstable and the Adam optimizer tends to outperform SQP in terms of the test loss, but it takes a 100 evaluations and the SQP converges much quicker to a high test accuracy. Finally, the L-BFGS also exhibits poor performance on the \texttt{conv} model, SQP converges to a lower classification loss much faster. Adam reaches a higher test accuracy, but also a higher classification loss.

\subsection{Hyperparameter Optimization}

\begin{figure}[!t]
\centering
\includegraphics[width=\linewidth]{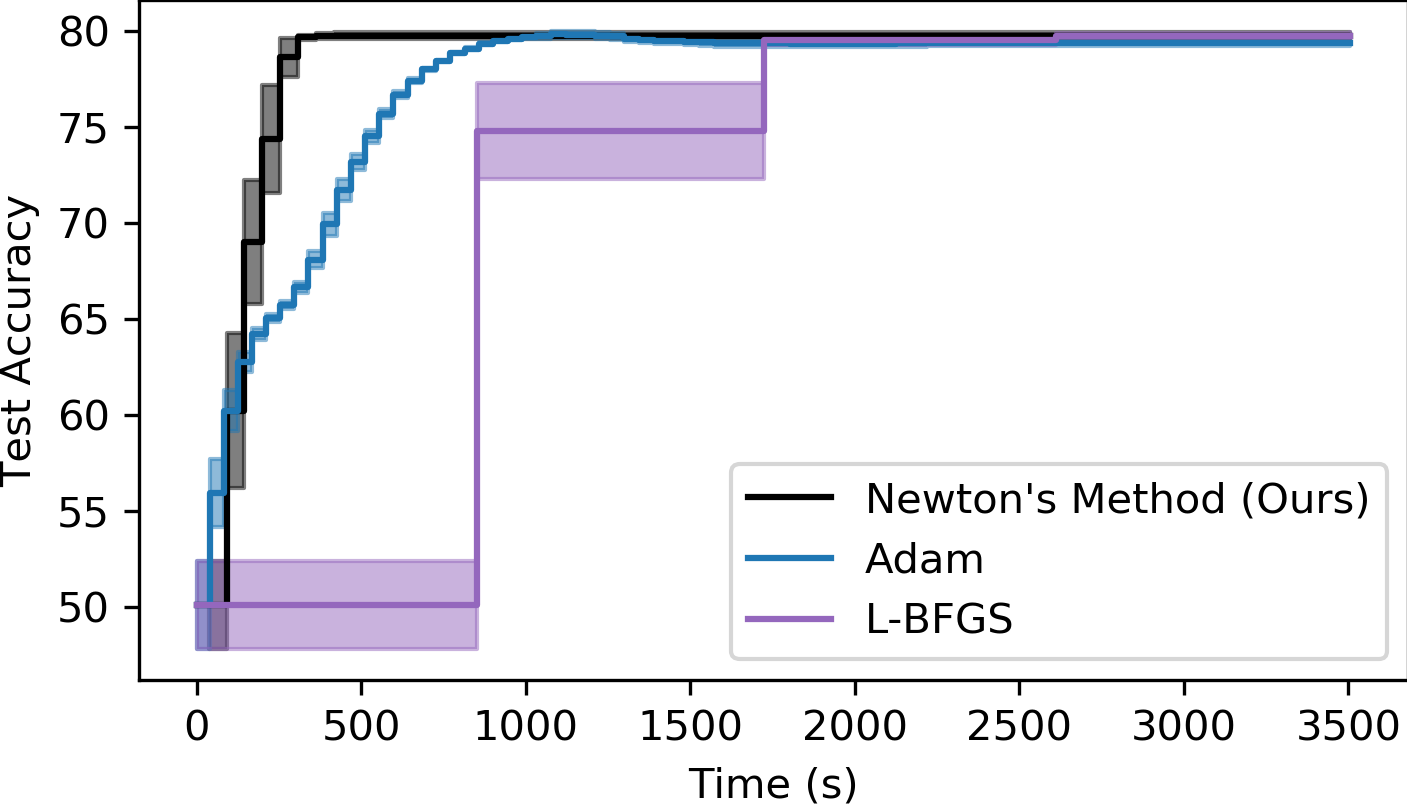}
\caption{Optimization of multi-class SVM on Fashion-MNIST. Newton's method using second-order derivatives derived in this paper leads to much faster convergence.}
\label{fig:svm_opt}
\end{figure}

We further verify the usefulness of computing second-order derivatives in practical problem instances by applying the Newton's Method to tuning a multi-class SVM model \citep{crammer2001algorithmic} on the Fashion-MNIST \citep{xiao2017fashion} dataset for best performance on the test set. We reformulate the constraints as log-barrier penalty\footnote{The log-barrier is defined for a constraint $x \leq 0$ as $-\log\left(-\alpha x\right) / \alpha$ for a tunable refinement constant $\alpha$.} to ensure smoothness, with the refinement value of $\alpha = 10^2$. We solve the resulting problem with the Mosek optimizer.\footnote{\url{www.mosek.com}}

We scale the images down to 14 $\times$ 14 to aid with the memory requirements and fit a multi-class SVM model to 2,000 randomly selected samples in the train set. We define upper level loss as the cross-entropy classification loss between the predictions and the labels. We select 5 random seeds to verify algorithmic performance.

We show a test accuracy plot in Figure~\ref{fig:svm_opt}. We note that the optimization time, i.e., training the lower level multi-class SVM, dominates both the gradient and Hessian computation. The application of Newton's Method (enabled by second-order derivatives shown in this work) significantly reduces the number of necessary lower level optimizations and reaches maximum test accuracy much quicker in terms of wall-clock time.

\subsection{Parameter Loss Landscape in Inverse Optimal Control}\label{sec:ioc}

We investigate the Inverse Optimal Control (IOC) problem as a case study in loss landscape or curvature analysis. IOC has a natural formulation as a BLPP. 

IOC has a large body of literature of its own, e.g. \citet{keshavarz2011imputing, johnson2013inverse, terekhov2011analytical}, but we focus here on some simple examples that are illustrative to the general BLPP curvature analysis. We show that the unconstrained Optimal Control Problem (OCP) problem we formulate is globally convex in parameter, but that in the presence of constraints in the OCP, the resulting BLPP requires more care---linear inequalities, without reformulation, violate our smoothness assumption in Theorem~\ref{thm:hessian}.

Given known linear time-invariant discrete time dynamics $x^{(i+1)} = A x^{(i)} + B u^{(i)}$ and state and control cost matrices $Q, R$, and control limits $u^{(i)} \in \mathbb{U} = \{u \mid \norm{u}_\infty \leq u_\text{lim}\}$, we assume we observe some expert's trajectory $X_\text{e} = [ x^{(0)}, \dots, x^{(N)} ]$ and control history $U_\text{e} = [ u^{(0)}, \dots, u^{(N)} ]$. We seek to learn a reference trajectory we assume the expert is tracking, i.e., the expert behaves optimally under the cost $J(X, U)$:
\begin{equation}
X_\text{e}, U_\text{e} = \argmin_{x^{(i+1)} = A x^{(i)} + B u^{(i)}, U \in \mathbb{U}} J(X, U, X_\text{e}, U_\text{e})
\end{equation}
where $J(X, U, X_\text{e}, U_\text{e}) = \sum_i (x^{(i)} - x^{(i)}_\text{e,ref})^T Q (x^{(i)} - x^{(i)}_\text{e,ref}) + (u^{(i)} - u^{(i)}_\text{e,ref})^T R (u^{(i)}- u^{(i)}_\text{e,ref})$.
This leads to a BLPP
\begin{equation}
\begin{aligned}
\min_{X_\text{ref}, U_\text{ref}} ~~~ & \norm{X_\text{e} - X^\star}_2^2 + \norm{U_\text{e} - U^\star} \\
\text{s.t.} ~~~ X^\star,&\ U^\star = \argmin_{\substack{x^{(i+1)} = A x^{(i)} + B u^{(i)}, \\ \!\!\!\!\!\!U \in \mathbb{U}}} J(X, U, X_\text{ref}, U_\text{ref}). 
\end{aligned}
\end{equation}
In the typical notation of this paper, $\zs = \big(X^\star, U^\star\big)$, $p = \big(X_\text{ref}, U_\text{ref}\big)$. The problem corresponds to discovering the reference trajectory of an optimal agent, e.g., the centerline of a road using an observed trajectory of an autonomous car.

We show the comparison between the upper loss landscapes for IOC with unconstrained/constrained controls in Figure~\ref{fig:ioc_comp}. We employ the Principle Component Analysis (PCA) dimension reduction technique on the optimization path to visualize a many dimensional upper loss in 2 dimensions, inspired by \citet{li2017visualizing}.

In the absence any constraints, the resulting problem as stated is convex which follows from the application of Equations \eqref{eq:chain_2nd} and \eqref{eq:2nd_hess}.

\begin{figure}[!t]
\centering
\subfigure[Unconstrained]{
\includegraphics[width=0.46\linewidth]{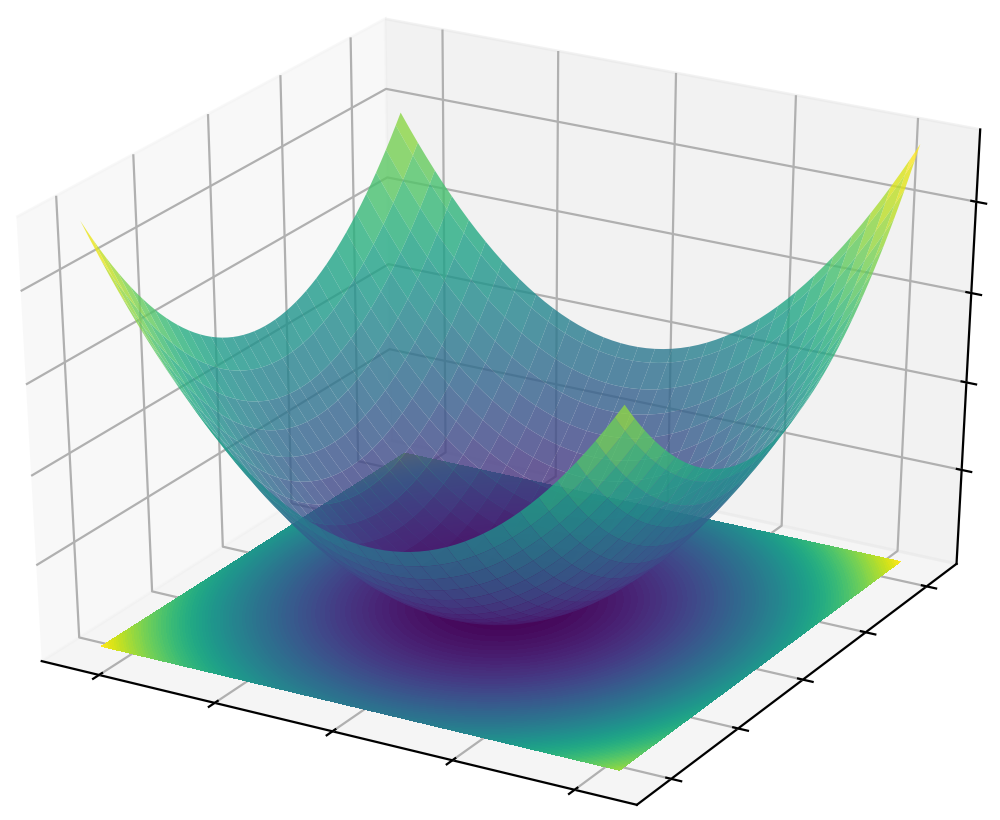}
}
\subfigure[Control Constrained]{
\includegraphics[width=0.46\linewidth]{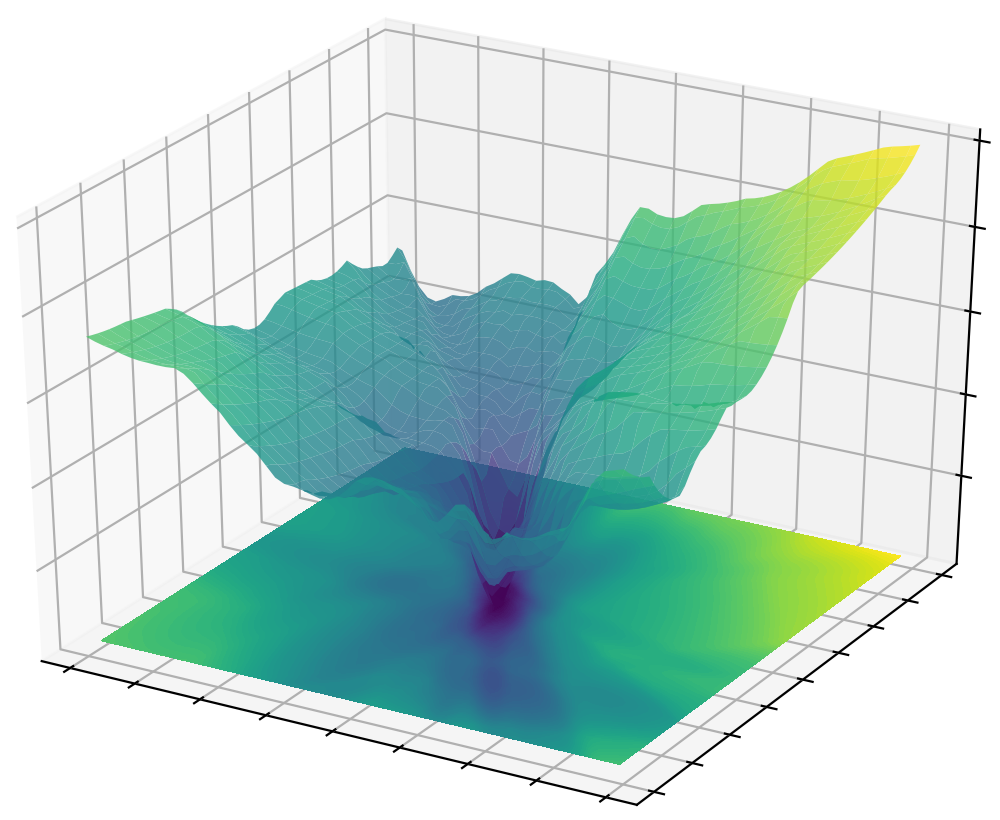}
}
\caption{Comparison of upper loss landscape in inverse optimal control without and with control constraints. Surface and the contour plots show the PCA 2D projections based on the optimization path.}
\label{fig:ioc_comp}
\end{figure}

\begin{figure}[!b]
\centering
\subfigure[Log-barrier, $\alpha = 10^1$]{
\includegraphics[width=0.46\linewidth]{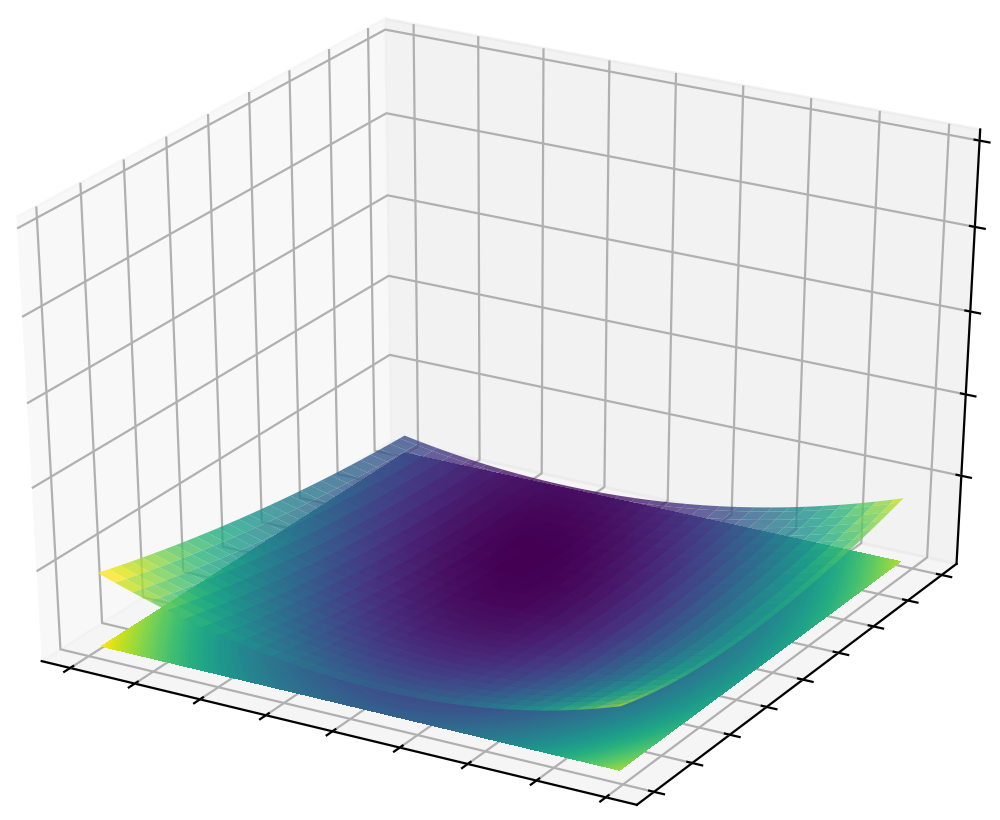}
}
\subfigure[Log-barrier, $\alpha = 10^{2.5}$]{
\includegraphics[width=0.46\linewidth]{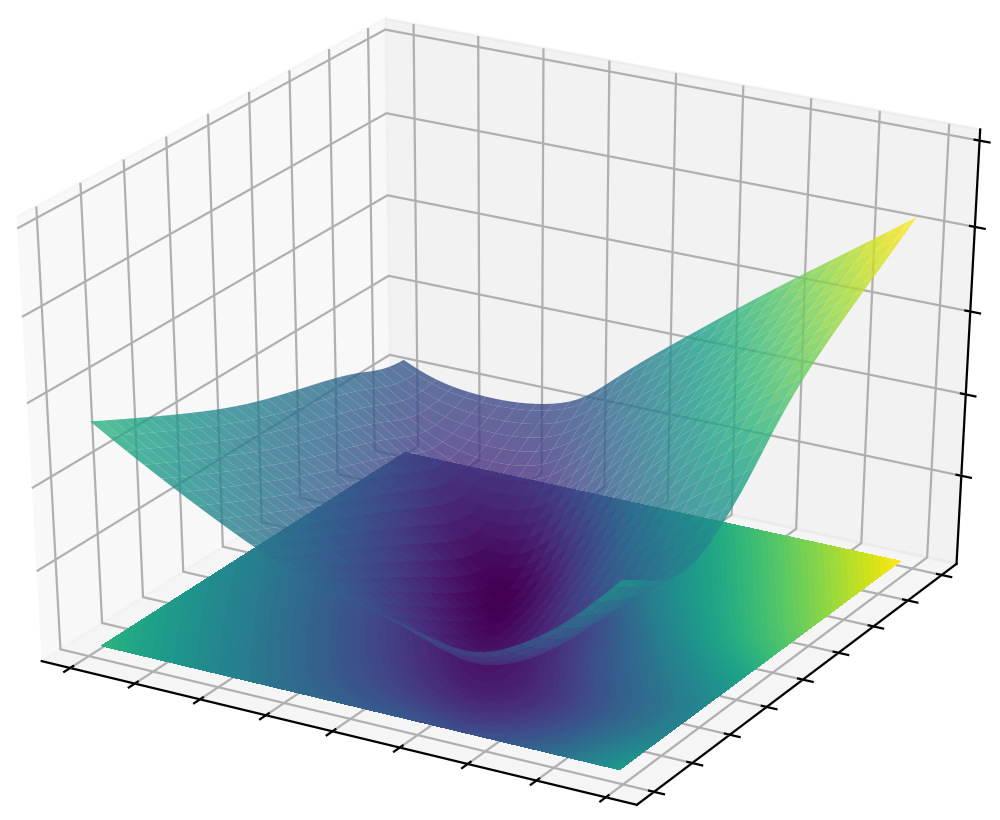}
}
\subfigure[Log-barrier, $\alpha = 10^4$]{
\includegraphics[width=0.46\linewidth]{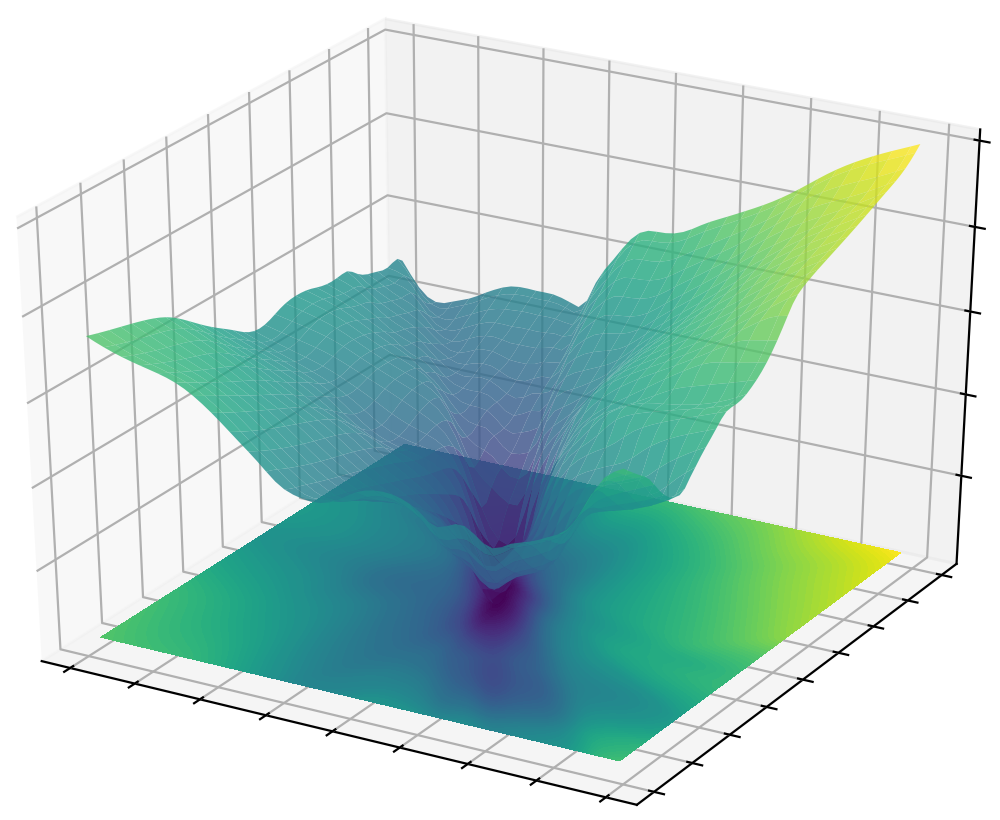}
}
\subfigure[Exactly constrained]{
\includegraphics[width=0.46\linewidth]{figs/mpc/mpc.png}
}
\caption{Comparison of loss landscape for inverse box-constrained-control MPC. The surface and contour plots show the PCA 2D projections based on the optimization path.}
\label{fig:mpc_landscape}
\end{figure}

Naive application of Equation~\eqref{eq:2nd_hess} to a constrained IOC problem yields a globally positive definite Hessian, but violates the assumptions of Theorem~\ref{thm:hessian}, since FOOC with linear constraints are not twice continuously differentiable. In the presence of even the simple maximum control value constraints considered, the loss-in-parameter (the upper loss) is highly non-convex. We observe that any constraints can be smoothly approximated to any degree of precision using the log-barrier function which is infinitely differentiable. Figure~\ref{fig:mpc_landscape} shows the 2D projected loss-in-parameter $p$ landscape for a successive refinement of the log-barrier constraints; for high values of refinement $\alpha \rightarrow \infty$, the loss landscape closely approximates the exactly constrained version, yet remains differentiable, so Theorem~\ref{thm:hessian} can be applied. 

\paragraph{Code release}
The code for our experiments is contained at \url{https://github.com/StanfordASL/Second-OrderSensitivityAnalysisForBilevelOptimization}.

\section{DISCUSSION}\label{sec:discussion}

\paragraph{Limitations \& Promises}
The method we propose here offers to make better use of every single lower level problem evaluation, but comes with limitations. Firstly, like any second-order optimization method, it might not be best suited for highly non-convex landscapes, which can be common in BLPPs as Figure \ref{fig:ioc_comp} shows. Secondly, in computational complexity analysis we do not focus on obtaining partial explicit derivatives necessary to compute the second-order sensitivity expression. We do so because these can be problem specific and often be zero or have an analytic form, but if they are obtained via automatic differentiation (the method we employ), their computation time highly depends on chosen software package. In general, it might turn out that computing second-order sensitivity information might take longer than several evaluations of the lower level and the first-order sensitivity information evaluation---at which point gradient-only optimization methods will likely function better than our proposed approach. Nevertheless, our work expands the optimization toolbox where some examples of BLPPs can be optimized much quicker, as Figure \ref{fig:svm_opt} shows.

\section{CONCLUSIONS}\label{sec:conclusions}

In this work we derive the second-order derivatives of the upper level objective in a general bilevel program via sensitivity analysis using the implicit function theorem. Second-order information enables second-order optimization to be applied to these problems, which we argue can drastically reduce the number of lower level function evaluations, and speed up optimization, in cases where the lower level evaluation dominates. We further show that the computational complexity of our proposed approach is comparable to first-order IFT and we adapt  error bound analysis for first-order IFT derivatives to our second-order IFT derivatives. 

\paragraph{Future Work} Future work includes quantifying how well various approximations suggested for first-order IFT apply to second-order IFT and whether (and when), for optimization purposes, some terms in our second-order sensitivity expression can be omitted or approximated to make their computation quicker. Finally, we are interested in further investigating loss landscapes, and the associated difficulty of optimization, for constrained bilevel problems. 


\subsection*{Acknowledgments}
Toyota Research Institute and the NASA University Leadership Initiative (grant \#80NSSC20M0163) provided funds to support this work; this article solely reflects the opinions and conclusions of its authors and not any Toyota or NASA entity.

\nocite{*}
\bibliography{bibliography.bib}  

\begin{thebibliography}{}

\bibitem[Agrawal et~al., 2019a]{agrawal2019layers}
Agrawal, A., Amos, B., Barratt, S., Boyd, S., Diamond, S., and Kolter, Z.
  (2019a).
\newblock Differentiable convex optimization layers.
\newblock {\em arXiv preprint arXiv:1910.12430}.

\bibitem[Agrawal et~al., 2019b]{agrawal2019cone}
Agrawal, A., Barratt, S., Boyd, S., Busseti, E., and Moursi, W.~M. (2019b).
\newblock Differentiating through a cone program.
\newblock {\em arXiv preprint arXiv:1904.09043}.

\bibitem[Ailon and Chazelle, 2006]{ailon2006approximate}
Ailon, N. and Chazelle, B. (2006).
\newblock Approximate nearest neighbors and the fast johnson-lindenstrauss
  transform.
\newblock In {\em Proceedings of the thirty-eighth annual ACM symposium on
  Theory of computing}, pages 557--563.

\bibitem[Amos and Kolter, 2017]{amos2017optnet}
Amos, B. and Kolter, J.~Z. (2017).
\newblock Optnet: Differentiable optimization as a layer in neural networks.
\newblock In {\em International Conference on Machine Learning}, pages
  136--145. PMLR.

\bibitem[Amos et~al., 2018]{amos2018differentiable}
Amos, B., Rodriguez, I. D.~J., Sacks, J., Boots, B., and Kolter, J.~Z. (2018).
\newblock Differentiable mpc for end-to-end planning and control.
\newblock {\em arXiv preprint arXiv:1810.13400}.

\bibitem[Andrychowicz et~al., 2016]{andrychowicz2016learning}
Andrychowicz, M., Denil, M., Gomez, S., Hoffman, M.~W., Pfau, D., Schaul, T.,
  Shillingford, B., and De~Freitas, N. (2016).
\newblock Learning to learn by gradient descent by gradient descent.
\newblock In {\em Advances in neural information processing systems}, pages
  3981--3989.

\bibitem[Bard, 1984]{bard1984optimality}
Bard, J.~F. (1984).
\newblock Optimality conditions for the bilevel programming problem.
\newblock {\em Naval research logistics quarterly}, 31(1):13--26.

\bibitem[Bard, 2013]{bard2013practical}
Bard, J.~F. (2013).
\newblock {\em Practical bilevel optimization: algorithms and applications},
  volume~30.
\newblock Springer Science \& Business Media.

\bibitem[Barratt, 2018]{barratt2018differentiability}
Barratt, S. (2018).
\newblock On the differentiability of the solution to convex optimization
  problems.
\newblock {\em arXiv preprint arXiv:1804.05098}.

\bibitem[Barratt and Boyd, 2021]{barratt2021least}
Barratt, S.~T. and Boyd, S.~P. (2021).
\newblock Least squares auto-tuning.
\newblock {\em Engineering Optimization}, 53(5):789--810.

\bibitem[Bengio, 2000]{bengio2000gradient}
Bengio, Y. (2000).
\newblock Gradient-based optimization of hyperparameters.
\newblock {\em Neural computation}, 12(8):1889--1900.

\bibitem[Bertinetto et~al., 2018]{bertinetto2018meta}
Bertinetto, L., Henriques, J.~F., Torr, P.~H., and Vedaldi, A. (2018).
\newblock Meta-learning with differentiable closed-form solvers.
\newblock {\em arXiv preprint arXiv:1805.08136}.

\bibitem[Bertrand et~al., 2020]{bertrand2020implicit}
Bertrand, Q., Klopfenstein, Q., Blondel, M., Vaiter, S., Gramfort, A., and
  Salmon, J. (2020).
\newblock Implicit differentiation of lasso-type models for hyperparameter
  optimization.
\newblock In {\em International Conference on Machine Learning}, pages
  810--821. PMLR.

\bibitem[Blondel et~al., 2021]{blondel2021efficient}
Blondel, M., Berthet, Q., Cuturi, M., Frostig, R., Hoyer, S.,
  Llinares-L{\'o}pez, F., Pedregosa, F., and Vert, J.-P. (2021).
\newblock Efficient and modular implicit differentiation.
\newblock {\em arXiv preprint arXiv:2105.15183}.

\bibitem[Crammer and Singer, 2001]{crammer2001algorithmic}
Crammer, K. and Singer, Y. (2001).
\newblock On the algorithmic implementation of multiclass kernel-based vector
  machines.
\newblock {\em Journal of machine learning research}, 2(Dec):265--292.

\bibitem[Diamond and Boyd, 2016]{diamond2016cvxpy}
Diamond, S. and Boyd, S. (2016).
\newblock {CVXPY}: {A} {P}ython-embedded modeling language for convex
  optimization.
\newblock {\em Journal of Machine Learning Research}, 17(83):1--5.

\bibitem[El~Ghaoui, 2002]{el2002inversion}
El~Ghaoui, L. (2002).
\newblock Inversion error, condition number, and approximate inverses of
  uncertain matrices.
\newblock {\em Linear algebra and its applications}, 343:171--193.

\bibitem[Finn et~al., 2017]{finn2017model}
Finn, C., Abbeel, P., and Levine, S. (2017).
\newblock Model-agnostic meta-learning for fast adaptation of deep networks.
\newblock In {\em International Conference on Machine Learning}, pages
  1126--1135. PMLR.

\bibitem[Golovin et~al., 2017]{golovin2017google}
Golovin, D., Solnik, B., Moitra, S., Kochanski, G., Karro, J., and Sculley, D.
  (2017).
\newblock Google vizier: A service for black-box optimization.
\newblock In {\em Proceedings of the 23rd ACM SIGKDD international conference
  on knowledge discovery and data mining}, pages 1487--1495.

\bibitem[Golub and Van~Loan, 1996]{golub1996matrix}
Golub, G.~H. and Van~Loan, C.~F. (1996).
\newblock Matrix computations. johns hopkins studies in the mathematical
  sciences.

\bibitem[Gould et~al., 2016]{gould2016differentiating}
Gould, S., Fernando, B., Cherian, A., Anderson, P., Cruz, R.~S., and Guo, E.
  (2016).
\newblock On differentiating parameterized argmin and argmax problems with
  application to bi-level optimization.
\newblock {\em arXiv preprint arXiv:1607.05447}.

\bibitem[Gruber, 2017]{gruber2017improving}
Gruber, M.~H. (2017).
\newblock {\em Improving efficiency by shrinkage: the James-Stein and ridge
  regression estimators}.
\newblock Routledge.

\bibitem[Harrison et~al., 2018]{harrison2018meta}
Harrison, J., Sharma, A., and Pavone, M. (2018).
\newblock Meta-learning priors for efficient online bayesian regression.
\newblock {\em arXiv preprint arXiv:1807.08912}.

\bibitem[Henrion and Surowiec, 2011]{henrion2011calmness}
Henrion, R. and Surowiec, T. (2011).
\newblock On calmness conditions in convex bilevel programming.
\newblock {\em Applicable Analysis}, 90(6):951--970.

\bibitem[Higham, 2002]{higham2002accuracy}
Higham, N.~J. (2002).
\newblock {\em Accuracy and stability of numerical algorithms}.
\newblock SIAM.

\bibitem[Johnson et~al., 2013]{johnson2013inverse}
Johnson, M., Aghasadeghi, N., and Bretl, T. (2013).
\newblock Inverse optimal control for deterministic continuous-time nonlinear
  systems.
\newblock In {\em 52nd IEEE Conference on Decision and Control}, pages
  2906--2913. IEEE.

\bibitem[Keshavarz et~al., 2011]{keshavarz2011imputing}
Keshavarz, A., Wang, Y., and Boyd, S. (2011).
\newblock Imputing a convex objective function.
\newblock In {\em 2011 IEEE international symposium on intelligent control},
  pages 613--619. IEEE.

\bibitem[Kingma and Ba, 2014]{kingma2014adam}
Kingma, D.~P. and Ba, J. (2014).
\newblock Adam: A method for stochastic optimization.
\newblock {\em arXiv preprint arXiv:1412.6980}.

\bibitem[LeCun, 1998]{lecun1998mnist}
LeCun, Y. (1998).
\newblock The mnist database of handwritten digits.

\bibitem[Lee et~al., 2019]{lee2019meta}
Lee, K., Maji, S., Ravichandran, A., and Soatto, S. (2019).
\newblock Meta-learning with differentiable convex optimization.
\newblock In {\em Proceedings of the IEEE/CVF Conference on Computer Vision and
  Pattern Recognition}, pages 10657--10665.

\bibitem[Li et~al., 2017]{li2017visualizing}
Li, H., Xu, Z., Taylor, G., Studer, C., and Goldstein, T. (2017).
\newblock Visualizing the loss landscape of neural nets.
\newblock {\em arXiv preprint arXiv:1712.09913}.

\bibitem[Liu and Nocedal, 1989]{liu1989limited}
Liu, D.~C. and Nocedal, J. (1989).
\newblock On the limited memory bfgs method for large scale optimization.
\newblock {\em Mathematical programming}, 45(1):503--528.

\bibitem[Liu et~al., 2001]{liu2001exact}
Liu, G., Han, J., and Zhang, J. (2001).
\newblock Exact penalty functions for convex bilevel programming problems.
\newblock {\em Journal of Optimization Theory and Applications},
  110(3):621--643.

\bibitem[Lorraine et~al., 2020]{lorraine2020optimizing}
Lorraine, J., Vicol, P., and Duvenaud, D. (2020).
\newblock Optimizing millions of hyperparameters by implicit differentiation.
\newblock In {\em International Conference on Artificial Intelligence and
  Statistics}, pages 1540--1552. PMLR.

\bibitem[Magnus and Neudecker, 2019]{magnus2019matrix}
Magnus, J.~R. and Neudecker, H. (2019).
\newblock {\em Matrix differential calculus with applications in statistics and
  econometrics}.
\newblock John Wiley \& Sons.

\bibitem[Mehlitz and Zemkoho, 2021]{mehlitz2021sufficient}
Mehlitz, P. and Zemkoho, A.~B. (2021).
\newblock Sufficient optimality conditions in bilevel programming.
\newblock {\em Mathematics of Operations Research}.

\bibitem[Nocedal and Wright, 2006]{nocedal2006numerical}
Nocedal, J. and Wright, S. (2006).
\newblock {\em Numerical optimization}.
\newblock Springer Science \& Business Media.

\bibitem[Rajeswaran et~al., 2019]{rajeswaran2019meta}
Rajeswaran, A., Finn, C., Kakade, S., and Levine, S. (2019).
\newblock Meta-learning with implicit gradients.

\bibitem[Sinha et~al., 2017]{sinha2017review}
Sinha, A., Malo, P., and Deb, K. (2017).
\newblock A review on bilevel optimization: from classical to evolutionary
  approaches and applications.
\newblock {\em IEEE Transactions on Evolutionary Computation}, 22(2):276--295.

\bibitem[Terekhov and Zatsiorsky, 2011]{terekhov2011analytical}
Terekhov, A.~V. and Zatsiorsky, V.~M. (2011).
\newblock Analytical and numerical analysis of inverse optimization problems:
  conditions of uniqueness and computational methods.
\newblock {\em Biological cybernetics}, 104(1):75--93.

\bibitem[Tikhonov, 1943]{tikhonov1943stability}
Tikhonov, A.~N. (1943).
\newblock On the stability of inverse problems.
\newblock In {\em Dokl. Akad. Nauk SSSR}, volume~39, pages 195--198.

\bibitem[Von~Stackelberg, 2010]{von2010market}
Von~Stackelberg, H. (2010).
\newblock {\em Market structure and equilibrium}.
\newblock Springer Science \& Business Media.

\bibitem[Wachsmuth, 2014]{wachsmuth2014differentiability}
Wachsmuth, G. (2014).
\newblock Differentiability of implicit functions: Beyond the implicit function
  theorem.
\newblock {\em Journal of Mathematical Analysis and Applications},
  414(1):259--272.

\bibitem[Wiesemann et~al., 2013]{wiesemann2013pessimistic}
Wiesemann, W., Tsoukalas, A., Kleniati, P.-M., and Rustem, B. (2013).
\newblock Pessimistic bilevel optimization.
\newblock {\em SIAM Journal on Optimization}, 23(1):353--380.

\bibitem[Xiao et~al., 2017]{xiao2017fashion}
Xiao, H., Rasul, K., and Vollgraf, R. (2017).
\newblock Fashion-mnist: a novel image dataset for benchmarking machine
  learning algorithms.
\newblock {\em arXiv preprint arXiv:1708.07747}.

\bibitem[Yao et~al., 2020]{yao2020adahessian}
Yao, Z., Gholami, A., Shen, S., Mustafa, M., Keutzer, K., and Mahoney, M.~W.
  (2020).
\newblock Adahessian: An adaptive second order optimizer for machine learning.
\newblock {\em arXiv preprint arXiv:2006.00719}.

\end{thebibliography}

\onecolumn

\hsize\textwidth
\linewidth\hsize \toptitlebar {\centering
  {\Large\bfseries Second-Order Sensitivity Analysis for Bilevel Optimization\\
Appendix \par}}
\bottomtitlebar%


\setcounter{section}{0}
\renewcommand{\thesection}{\Alph{section}}
\renewcommand*{\theHsection}{chX.\the\value{section}}
\section{PROOFS}
\label{sec:appendix}

\setcounter{thm}{1}
\begin{thm}[Second-Order IFT]\label{thm:hessian}
Let $k: \mathbb{R}^{m} \times \mathbb{R}^{n} \rightarrow \mathbb{R}^m$ be a twice continuously differentiable multivariate function of two variables $\zs \in \mathbb{R}^m$ and $p \in \mathbb{R}^n$ such that $k$ defines a fixed point for $x$, i.e. $k(\zs, p) = 0$ for all values of $p \in \mathbb{R}^n$. Then, the Hessian of (implicitly defined) $\zs$ w.r.t. $p$ is given by
\begin{equation}\label{eq:2nd_hess}
\begin{aligned}
H_p \zs = -\bigg[ \big(D_\zs k \big)^{-1} \otimes I \bigg] \bigg[ 
H_{p} k + \big( D_{p \zs} k \big) \big( D_p \zs \big) \\
~~~~~~~~~~~~+ \big( I \otimes \big(D_p \zs\big)^T \big) \big( D_{\zs p} k \big) \\
~~~~~~~~~~~~+ \big( I \otimes \big(D_p \zs\big)^T \big) \big( H_{\zs} k \big) \big( D_p \zs \big) \bigg].
\end{aligned}\tag{7}
\end{equation}
\end{thm}

\begin{proof}
Differentiation of the implicit expression further requires establishing a convention for matrix expression with vector inputs derivatives. We keep the standard convention of representing the partial Jacobian (with operator $D$) and scalar function Hessian (with operator $H$). The missing representations include the Hessian and the 2nd order mixed derivative. Let $f : \R{m} \times \R{n} \mapsto \R{p}$, then we represent Hessian and 2nd order mixed derivatives as 
\begin{equation}
\begin{aligned}
D_{a b} f(a, b) &= \begin{bmatrix}
D_b \big(\nabla_a \big(f_1 \big) \big) \\
\vdots \\
D_b \big(\nabla_a \big(f_{p}\big) \big) \\
\end{bmatrix} &\in \R{pm \times n}
\\
H_a f(a, b) &= \begin{bmatrix}
H_a f_1 \\
\vdots \\
H_a f_{p} \\
\end{bmatrix} &\in \R{pm \times m}.
\end{aligned}
\end{equation}

We introduce the following derivatives rules
\begin{align}
D_a \big( A(a) B \big) &= \big(I \otimes B^T\big) \big( D_a A(a) \big)\\
D_a \big( B A(a) \big) &= \big(B \otimes I\big) \big(D_a A(a) \big)
\end{align}
where the chain rule applies as
\begin{equation}
D_a \big(A(c(a)\big) = \big(D_c A(c)\big) \big(D_a c(a)\big)
\end{equation}

This allows to differentiate $D_p k + \big( D_\zs k \big) \big( D_p \zs \big) = 0$ giving
\begin{equation}
\begin{aligned}
\label{eq:2nd}
& H_{p} k(\zs, p) \\
&+ \big( D_{p \zs} k(\zs, p) \big) \big( D_p \zs \big) \\
&+ \big( I \otimes \big(D_p \zs\big)^T \big) \big( D_{\zs p} k(\zs, p) \big) \\
&+ \big( I \otimes \big(D_p \zs\big)^T \big) \big( H_{\zs} k(\zs, p) \big) \big( D_p \zs \big) \\
&+ \big( \big(D_\zs k(\zs, p) \big) \otimes I \big) \big( H_{p} \zs \big) \\
&= 0
\end{aligned}
\end{equation}

The implicit Hessian is given by solving Equation~\eqref{eq:2nd} for $H_{p} \zs$.
\begin{equation}
\begin{aligned}
H_p \zs = -\bigg[ \big(D_\zs k(\zs, p) \big)^{-1} \otimes I \bigg] \bigg[ 
& H_{p} k(\zs, p) + \big( D_{p \zs} k(\zs, p) \big) \big( D_p \zs \big) & \\
& + \big( I \otimes \big(D_p \zs\big)^T \big) \big( D_{\zs p} k(\zs, p) \big) & \\
& + \big( I \otimes \big(D_p \zs\big)^T \big) \big( H_{\zs} k(\zs, p) \big) \big( D_p \zs \big) \bigg] & 
\end{aligned}\tag{7}
\end{equation}
exploiting the identity $\big( A \otimes I \big)^{-1} = A^{-1} \otimes I$.
\end{proof}


\begin{thm}[First-Order Error Bound]\label{thm:error1}
Given the result of IFT applied to an exact solution $D_\zs p = g = -A^{-1} B$, where $A(\zs, p) = D_\zs k(\zs, p)$, $B(\zs, p) = D_p k(\zs, p)$ and an inexact solution $\widetilde{A}(z, p) = D_z k(z, p)$ (where $k(z, p) \neq 0$) and $\widetilde{B}(z, p) = D_p k(z, p)$. Assume $\norm{z - \zs} \leq \delta$,
$\norm{\widetilde{A} - A}_\text{op} \leq \gamma \delta$,
$\norm{\widetilde{B} - B}_F \leq \beta \delta$,
$\norm{B}_F \leq R$,
$\norm{\widetilde{A} v} \geq \alpha_1 \norm{v}$, $\norm{A v} \geq \alpha_2 \norm{v}$,
then
\begin{equation}
\norm{\widetilde{J} - J}_F \leq \frac{\beta}{\alpha_1} \delta + \frac{\gamma R}{\alpha_1 \alpha_2} \delta
\end{equation}
\end{thm}

\begin{proof}
\begin{align}
-(\widetilde{J} - J) &= \widetilde{A}^{-1} \widetilde{B} - A^{-1} B \\
&= \widetilde{A}^{-1} \widetilde{B} - \widetilde{A}^{-1} B + \widetilde{A}^{-1} B - A^{-1} B \\
&= \widetilde{A}^{-1} (\widetilde{B} - B) + (\widetilde{A}^{-1} - A^{-1}) B \label{eq:alpha_bound}
\end{align}
which allows to bound the implicit gradient error as
\begin{align*}
\norm{\widetilde{J} - J}_F &\leq \norm{\widetilde{A}^{-1}}_\text{op} \norm{\widetilde{B} - B}_F + \norm{\widetilde{A}^{-1} - A^{-1}}_\text{op} \norm{B}_F \\
&\leq \norm{\widetilde{A}^{-1}}_\text{op} \norm{\widetilde{B} - B}_F + \norm{\widetilde{A}^{-1}}_\text{op} \norm{\widetilde{A} - A}_\text{op} \norm{A^{-1}}_\text{op} \norm{B}_F \\
& \leq \frac{\beta}{\alpha_1} \delta + \frac{\gamma R}{\alpha_1 \alpha_2} \delta  
\end{align*}
exploiting the fact that for any invertible matrices $M_1$, $M_2$ (where $\widetilde{A}$, $A$ are invertible from $\norm{\widetilde{A} v} \geq \alpha_1 \norm{v}$, $\norm{A v} \geq \alpha_2 \norm{v}$) $\big( M_1^{-1} - M_2^{-2}\big) = M_1^{-1} \big( M_1 - M_2 \big) M_2^{-1}$.
\end{proof}

\begin{thm}[Second-Order Error Bound]\label{thm:error2}
Given the result of IFT applied to an exact solution, for $\zs \in \mathbb{R}^m$, $H_\zs p = H = -A^{-1} B$, where $A(\zs, p) = D_\zs k(\zs, p) \otimes I$, $B(\zs, p) = H_{p} k + \big( D_{p \zs} k \big) \big( D_p \zs \big)
+ \big( I \otimes \big(D_p \zs\big)^T \big) \big( D_{\zs p} k \big) 
+ \big( I \otimes \big(D_p \zs\big)^T \big) \big( H_{\zs} k \big) \big( D_p \zs \big)$ and an inexact solution $\widetilde{A}(z, p)$, $\widetilde{B}(z, p)$ (where $k(z, p) \neq 0$). Assume $\norm{z - \zs} \leq \delta$, 
$\norm{\widetilde{A} - A}_\text{op} \leq \gamma \delta$,
$\norm{H_p k(z, p) - H_p k(\zs, p)}_F \leq \zeta \delta$,
$\norm{D_{z p} k(z, p) - D_{\zs p} k(\zs, p)}_F \leq \eta \delta$,
$\norm{H_z k(z, p) - H_\zs k(\zs, p)}_F \leq \nu \delta$,
$\norm{D_z p - D_\zs p}_F \leq \kappa_J \delta$ defined in Theorem \ref{thm:error1},
$\norm{B}_F \leq R_H$,
$\norm{\widetilde{A} v} \geq \alpha_1 \norm{v}$, $\norm{A v} \geq \alpha_2 \norm{v}$,
then
\begin{equation}
\norm{\widetilde{H} - H}_F \leq \frac{\zeta + 2 \eta \kappa_J + \nu \kappa_J^2}{\alpha_1} \delta + \frac{\gamma R_H}{\alpha_1 \alpha_2} \delta.
\end{equation}
\end{thm}

\begin{proof}
\begin{align}
-(\widetilde{H} - H) &= \left( \widetilde{A}^{-1} \otimes I \right) \widetilde{B} - \left( A^{-1} \otimes I \right) B \\
&= \left( \widetilde{A}^{-1} \otimes I \right) \widetilde{B} - \left( \widetilde{A}^{-1} \otimes I \right) B + \left( \widetilde{A}^{-1} \otimes I \right) B - \left( A^{-1} \otimes I \right) B\\
&= \left( \widetilde{A}^{-1} \otimes I \right) (\widetilde{B} - B) + \left( (\widetilde{A}^{-1} - A^{-1}) \otimes I \right) B \label{eq:alpha_bound} \\
\end{align}
which allows to bound the implicit Hessian error as
\begin{align*}
\norm{\widetilde{H} - H}_F &\leq \norm{\widetilde{A}^{-1}}_\text{op} \norm{\widetilde{B} - B}_F + \norm{\widetilde{A}^{-1} - A^{-1}}_\text{op} \norm{B} \\
&\leq \norm{\widetilde{A}^{-1}}_\text{op} \norm{\widetilde{B} - B}_F + \norm{\widetilde{A}^{-1}}_\text{op} \norm{\widetilde{A} - A}_\text{op} \norm{A^{-1}}_\text{op} \norm{B}_F \\
& \leq \frac{\zeta + 2 \eta \kappa_J + \nu \kappa_J^2}{\alpha_1} \delta + \frac{\gamma R_H}{\alpha_1 \alpha_2} \delta
\end{align*}
exploiting the fact that for any invertible matrices $M_1$, $M_2$ (where $\widetilde{A}$, $A$ are invertible from $\norm{\widetilde{A} v} \geq \alpha_1 \norm{v}$, $\norm{A v} \geq \alpha_2 \norm{v}$) $\big( M_1^{-1} - M_2^{-2}\big) = M_1^{-1} \big( M_1 - M_2 \big) M_2^{-1}$.
\end{proof}


\begin{thm}[Regularized First-Order Error Bound]\label{thm:error1_reg}
Given the result of regularized IFT applied to an exact solution $D_\zs p = g = -A^{-1} B$, where $A(\zs, p) = D_\zs k(\zs, p)$, $B(\zs, p) = D_p k(\zs, p)$ and an inexact solution $\widetilde{A}(z, p) = D_z k(z, p) + \epsilon I$, $\widetilde{B}(z, p) = D_p k(z, p)$ (where $k(z, p) \neq 0$) . Assume
$\norm{z - \zs} \leq \delta$,
$\norm{\widetilde{A} - A}_\text{op} \leq \gamma \delta$,
$\norm{\widetilde{B} - B}_F \leq \beta \delta$,
$\norm{B}_F \leq R$,
$\norm{\widetilde{A} v} \geq \alpha_1 \norm{v}$ and $v^T \widetilde{A} v \geq 0$, so
$\norm{\widehat{A} v} \geq (\alpha_1 + \epsilon) \norm{v}$,
$\norm{A v} \geq \alpha_2 \norm{v}$
then
\begin{equation}
\norm{\widehat{J} - J}_F \leq \frac{\beta \delta}{\alpha_1 + \epsilon} + \frac{R\big(\gamma \delta + \epsilon\big)}{\big(\alpha_1 + \epsilon\big) \alpha_2}
\end{equation}
\end{thm}

\begin{proof}
We observe that from singular value decomposition
\begin{equation}
\norm{\widetilde{A} v} \geq \alpha_1 \norm{v}, v^T \widetilde{A} v \geq 0 \implies \norm{\widetilde{A} + \epsilon I}_\text{op} \geq \alpha_1 + \epsilon
\end{equation}
which gives
\begin{align*}
\norm{\widehat{J} - J}_F &\leq \norm{\widehat{A}^{-1}}_\text{op} \norm{\widetilde{B} - B} + \norm{\widehat{A}^{-1} - A^{-1}}_\text{op} \norm{B} \\
&\leq \norm{\widehat{A}^{-1}}_\text{op} \norm{\widetilde{B} - B} + \norm{\widehat{A}^{-1}}_\text{op} \norm{\widehat{A} - A}_\text{op} \norm{A^{-1}}_\text{op} \norm{B} \\
&\leq \norm{\widehat{A}^{-1}}_\text{op} \norm{\widetilde{B} - B} + \norm{\widehat{A}^{-1}}_\text{op} \norm{\widetilde{A} + \epsilon I - A}_\text{op} \norm{A^{-1}}_\text{op} \norm{B} \\
&\leq \norm{\widehat{A}^{-1}}_\text{op} \norm{\widetilde{B} - B} + \norm{\widehat{A}^{-1}}_\text{op} \left( \norm{\widetilde{A} - A}_\text{op} + \norm{\epsilon I}_\text{op}\right) \norm{A^{-1}}_\text{op} \norm{B} \\
& \leq \frac{\beta}{\alpha_1 + \epsilon} \delta + \frac{R}{\big(\alpha_1 + \epsilon\big) \alpha_2} \left( \gamma \delta + \epsilon \right)
\end{align*}
which gives
\begin{equation}
\norm{\widehat{J} - J}_F \leq \frac{\beta \delta}{\alpha_1 + \epsilon} + \frac{R\big(\gamma \delta + \epsilon\big)}{\big(\alpha_1 + \epsilon\big) \alpha_2}.
\end{equation}
\end{proof}

\newpage
\section{COMPUTATIONAL COMPLEXITY}\label{sec:app:comp_complex}

\paragraph{Evaluating $A \otimes I$ Efficiently}
The product $A \otimes I$ features in the Hessian expression
\begin{equation*}
H_p \zs = \left[ \big(D_\zs k \big)^{-1} \otimes I \right] \left[ \dots \right]
\end{equation*}

Given the expression $M = \big(A \otimes I \big) C$,  let $A \in \R{m \times n}$, $C \in \R{(n p) \times r}$, such that
\[
C = \begin{bmatrix}
C_1 \in \R{p \times r} \\
\vdots \\
C_n \in \R{p \times r}
\end{bmatrix}
\]

Let $M = \big(A \otimes I\big) C$ and define $\widetilde{C} \in \R{n \times p \times r}$ s.t. $\widetilde{C}_{i**} \in \R{p \times r} ~~ \forall i \in [1..n]$, $\widetilde{M} \in \R{m \times p \times r}$ s.t. $\widetilde{M}_{j**} \in \R{p \times r} ~~ \forall j \in [1..m]$---a 3-d representation of $C$ and $M$ where the stacked matrices in $C$, $M$ are concatenated along a third, first, dimension in $\widetilde{C}$, $\widetilde{M}$.

The operation can now be defined formally in Einstein summation notation as
\begin{equation*}
\big(\widetilde{M} \big)_{ijk} = \big(A\big)_{il} \big(\widetilde{C}\big)_{ljk}
\end{equation*}

Intuitively, the operation $M = \big(A \otimes I \big)$ corresponds to a matrix multiplication performed for every vector in $C$ built from $n$ elements, one from each $C_i$.

\paragraph{Evaluating $I \otimes B$ Efficiently}

The product $I \otimes B$ features in the Hessian expression
\begin{equation*}
H_p \zs = \left[ \dots \right] \left[ \dots + \big(I \otimes \big(D_p \zs \big)^T \big) \big( D_{\zs p} k\big) + \dots \right]
\end{equation*}

Given the expression $M = \big(I \otimes B \big) C$,  let $A \in \R{m \times n}$, $C \in \R{(p n) \times r}$, such that
\[
C = \begin{bmatrix}
C_1 \in \R{n \times r} \\
\vdots \\
C_p \in \R{n \times r}
\end{bmatrix}
\]

Let $M = \big(I \otimes B\big) C$ and define $\widetilde{C} \in \R{p \times n \times r}$ s.t. $\widetilde{C}_{i**} \in \R{n \times r} ~~ \forall i \in [1..n]$, $\widetilde{M} \in \R{p \times m \times r}$ s.t. $\widetilde{M}_{j**} \in \R{m \times r} ~~ \forall j \in [1..p]$---a 3-d representation of $C$ and $M$ where the stacked matrices in $C$, $M$ are concatenated along a third, first, dimension in $\widetilde{C}$, $\widetilde{M}$.

The operation can now be defined formally in Einstein summation notation as
\begin{equation*}
\big(\widetilde{M} \big)_{ijk} = \big(A\big)_{jl} \big(\widetilde{C}\big)_{ilk}
\end{equation*}

Intuitively, the operation $M = \big(I \otimes B \big)$ corresponds to a matrix multiplication performed by broadcasting $B$ and performing matrix multiplication of $B$ with every $p$ element of $C$, so that $B C_i ~~ \forall i \in [1..p]$---these products are then stacked together to form $M$.

\newpage
\section{ADDITIONAL INSIGHTS}
\paragraph{Convexity-in-parameter}

Global convexity-in-parameter is obtained if the second-order order derivative of the upper loss function w.r.t. to the parameter is globally positive semi-definite (PSD) while both the objective functions are globally twice continuously differentiable. Recalling the full derivative expressions
\begin{equation*}
\begin{aligned}
H f_U(\zs, p) = H_p f_U + \big(D_p \zs)^T \big(H_\zs f_U\big) \big(D_p \zs) + \\
+ \big(\big(D_\zs f_U\big) \otimes I\big) H_p \zs
\end{aligned}
\end{equation*}
we conclude that it is not possible to guarantee the global PSD property for this Hessian in general, because of the reduction $\big(\big(D_\zs f_U\big) \otimes I\big) H_p \zs$ as outlined in Section \ref{sec:app:comp_complex}.\footnote{A stack of Hessian matrices are reduced by a weighted summation with unknown sign weights.} However, in the special case when $H_p \zs = 0$ and the problem is globally twice differentiable, global convexity is obtained, e.g., unconstrained quadratic programs with linear parameterizations.

\end{document}